\documentclass[preprint,1p,number,sort&compress]{elsarticle}

\usepackage{enumitem}
\usepackage[caption=false]{subfig}
\usepackage{graphicx}
\usepackage{amsmath,amsfonts,amsthm}

\usepackage{pgfplots}
\pgfplotsset{compat=newest}
\usetikzlibrary{plotmarks}
\usetikzlibrary{arrows.meta}
\usepgfplotslibrary{patchplots}
\usepackage{grffile}

\usepackage{hyperref}
\usepackage{cleveref}

\newtheorem{lemma}{Lemma}
\newtheorem{proposition}{Proposition}
\newtheorem{theorem}{Theorem}
\newtheorem{assumption}{Assumption}
\newtheorem{remark}{Remark}

\crefname{assumption}{assumption}{assumptions}

\crefname{remark}{remark}{remarks}

\crefname{equation}{}{} % remove eq. from \cref

\newcommand{\rme}{\mathrm{e}}

\newcommand{\norm}[1]{\lVert #1 \rVert}

\begin{document}
\myfooter[L]{}
\begin{frontmatter}
%\maketitle

\title{On the convergence of split exponential integrators
       for semilinear parabolic problems}

\author[1]{Marco Caliari}
\ead{marco.caliari@univr.it}
\author[1,2]{Fabio Cassini\corref{cor1}}
\ead{fabio.cassini@univr.it, cassini@altamatematica.it}
\author[3]{Lukas Einkemmer}
\ead{lukas.einkemmer@uibk.ac.at}
\author[3]{Alexander Ostermann}
\ead{alexander.ostermann@uibk.ac.at}
\affiliation[1]{organization={Department of Computer Science,
    University of Verona},addressline={Strada Le Grazie, 15},
  postcode={37134},
  city={Verona},
country={Italy}}
\affiliation[2]{organization={Istituto Nazionale di Alta Matematica},
  addressline={Piazzale Aldo Moro, 5},
  postcode={00185},
  city={Roma},
  country={Italy}}
\affiliation[3]{organization={Department of Mathematics,
    University of Innsbruck},addressline={Technikerstrasse, 13},
  postcode={A-6020},
  city={Innsbruck},
country={Austria}}

\begin{abstract}
Splitting the exponential-like $\varphi$ functions, which typically
appear in exponential integrators, is attractive in many situations
since it can dramatically reduce the computational cost of the procedure.
However, depending on the employed splitting, this can result in order
reduction.
The aim of this paper is to analyze different such split approximations.
We perform the analysis for semilinear problems in the abstract framework of
commuting semigroups and
derive error bounds that depend, in particular,
on whether the vector (to which the $\varphi$ functions are applied)
satisfies appropriate boundary conditions.
We then present the convergence
analysis for two split versions of a second-order
exponential Runge--Kutta integrator in the context of
analytic semigroups, and show that one suffers from order reduction while
the other does not.
Numerical results for semidiscretized parabolic PDEs
confirm the theoretical findings.
\end{abstract}

\end{frontmatter}

%\begin{MSCcodes}
%  65J08, 65L04, 65L70, 65M12
%\end{MSCcodes}

\section{Introduction}\label{sec:intro}

The effective numerical time integration of stiff evolution
differential equations is a highly challenging and widely
studied topic in numerical analysis. In this context,
exponential integrators have proven to efficient methods.
Indeed, the numerical solution of many physically relevant equations
has been computed by this kind of schemes.
We mention here, among others,
magnetohydrodynamic equations, models for meteorology
and elastodynamics, hyperbolic problems, and
various types of advection-diffusion-reaction,
Schr\"odinger  and complex Ginzburg--Landau 
equations
(see, e.g., \cite{CEM20,ETL17,HL03,KT05,LPR19,MLT17,ZZS23}).
Researchers have also put a lot of effort in the theoretical study of
integrators of exponential type, and we refer to the
seminal survey~\cite{HO10} for a comprehensive overview.

In this manuscript we are interested in semilinear
parabolic problems of the type
\begin{equation}\label{eq:PDE}
  \left\{
  \begin{aligned}
  \partial_t u(t)&=(A+B)u(t)+g(u(t)), \quad t\in (0,T],\\
u(0)&=u_0,
\end{aligned}
\right.
\end{equation}
where $u$ is the unknown,
$A$ and $B$ are differential
operators that \emph{commute} (for example,
directional operators like $A=\partial_{xx}$ and
$B=\partial_{yy}$ with homogeneous Dirichlet boundary conditions
on a square domain) and $g$ is
the nonlinear term.
Note that the operators in~\cref{eq:PDE} are always understood with boundary
conditions, and that the unknown, the initial condition and the nonlinear term
may also depend on the spatial variables
(the dependence is omitted for simplicity of notation).
The analytical solution of \cref{eq:PDE}
can be written by the variation-of-constants formula as
\begin{equation}\label{eq:voc}
  u(t)=\rme^{t(A+B)}u_0+
  \int_0^t\rme^{(t-s)(A+B)}g(u(s))ds,
\end{equation}
which is usually the starting point for the development of exponential
integrators.
The commutativity of the operators $A$ and $B$
implies
\begin{equation}\label{eq:expcomm}
  \rme^{t(A+B)}u_0=\rme^{tA}\rme^{tB}u_0=
  \rme^{tB}\rme^{tA}u_0.
\end{equation}
Note that if we consider $g(u(t))\equiv v$ in \cref{eq:PDE},
the solution of the (linear) problem can be written
as
\begin{equation}\label{eq:sollinear}
  u(t)=\rme^{tA}\rme^{tB}u_0+t\varphi_1(t(A+B))v,
  \quad \varphi_1(z) =
  \int_0^1\rme^{(1-\theta)z}d\theta.
  \end{equation}
Clearly, even if $A$ and $B$ do commute, the function $\varphi_1$ lacks
the semigroup property, and in general
\begin{equation}\label{eq:phicomm}
  t\varphi_1(t(A+B))v\ne t\varphi_1(tA)\varphi_1(tB)v=
  t\varphi_1(tB)\varphi_1(tA)v.
\end{equation}
On the other hand, when applying exponential integrators
to systems of stiff
ODEs arising from the semidiscretization in space of problems of
type \cref{eq:PDE},
it could be computationally attractive to employ split formulas
such as the
second (or the third) expression in \cref{eq:expcomm,eq:phicomm}
(see, for example, \cite{CC24bis,CCEOZ22,C24}).
Indeed, a classical way to recover an approximation
to solution \cref{eq:sollinear} is to apply the Strang splitting
scheme
to the exact flows of
\begin{equation*}
  \partial_t u_a(t)=Au_a(t),\quad
  \partial_t u_b(t)=Bu_b(t)+v,
\end{equation*}
obtaining at the first time step $t_1=\tau$
\begin{equation}\label{eq:Strang1}
  u(\tau) \approx u_1=\rme^{\tau A}\rme^{\tau B}u_0+
  \tau\rme^{\frac{\tau}{2}A}\varphi_1(\tau B)v.
\end{equation}
\begin{figure}[!htb]
  \centering
\begin{tikzpicture}

\begin{axis}[%
width=2.5in,
height=1.7in,
at={(0.758in,0.481in)},
scale only axis,
xmode=log,
xmin=1e-5,
xmax=1e-1,
xminorticks=true,
xlabel style={font=\color{white!15!black}},
xlabel={$\tau$},
ymode=log,
ymin=1e-11,
ymax=1e-2,
yminorticks=true,
ylabel style={font=\color{white!15!black}},
ylabel={Local error},
axis background/.style={fill=white},
legend style={at={(0.7,0.1)}, anchor=south west, legend cell align=left, align=left, draw=white!15!black},
legend columns = 1,
]

\addplot [color=magenta, only marks, line width=1pt, mark size=3.5pt, mark=asterisk, mark options={solid, magenta}]
  table[row sep=crcr]{%
0.015625	0.000196800256324142\\
0.00390625	6.59571613521527e-06\\
0.000976562499999999	1.43283760789688e-07\\
0.000244140625	2.54403616073831e-09\\
6.10351562499999e-05	4.18051814846592e-11\\
};
\addlegendentry{$p=3$}

\addplot [color=black, line width = 1pt, forget plot]
  table[row sep=crcr]{%
%0.015625	0.000701374559687327\\
%6.10351562499999e-05	4.18051814846592e-11\\
0.000976562499999999	3.582094019742200e-08\\
6.10351562499999e-05	1.045129537116480e-11\\
};
\addplot [color=cyan, only marks, line width=1pt, mark size = 3.5pt, mark=+, mark options={solid, cyan}]
  table[row sep=crcr]{%
0.015625	0.000138075014444178\\
0.00390625	1.06062873722825e-05\\
0.000976562499999999	6.68820123067853e-07\\
0.000244140625	4.10537942635501e-08\\
6.10351562499999e-05	2.41796032995906e-09\\
};
\addlegendentry{$p=2$}

%\addplot [color=black, dashed, line width = 1pt, forget plot]
%  table[row sep=crcr]{%
%0.015625	0.000158463448184197\\
%6.10351562499999e-05	2.41796032995906e-09\\
%};

\addplot [color=blue, only marks, line width=1pt, mark size= 3.5pt,mark=x, mark options={solid, blue}]
  table[row sep=crcr]{%
0.015625	5.35166573935856e-05\\
0.00390625	2.82727222487009e-06\\
0.000976562499999999	1.68582237085288e-07\\
0.000244140625	1.02720121362514e-08\\
6.10351562499999e-05	6.04527877615975e-10\\
};
\addlegendentry{$p=1$}
\addplot [color=black, dashed, line width=1pt,forget plot]
  table[row sep=crcr]{%
%0.015625	3.96183389874404e-05\\
%6.10351562499999e-05	6.04527877615975e-10\\
0.000976562499999999	 1.685822370852880e-06\\
6.10351562499999e-05	6.045278776159751e-09\\
};

\addplot [color=red, only marks, line width=1pt, mark size = 3.5pt, mark=o, mark options={solid, red}]
  table[row sep=crcr]{%
0.015625	0.00117872951576602\\
0.00390625	0.000294330494825977\\
0.000976562499999999	7.32604967307338e-05\\
0.000244140625	1.7962482105881e-05\\
6.103515625e-05	4.1755986256652e-06\\
};
\addlegendentry{$p=0$}

\addplot [color=black, dotted, line width=1pt, forget plot]
  table[row sep=crcr]{%
%0.015625	0.00106895324817029\\
%6.103515625e-05	4.1755986256652e-06\\
0.000976562499999999	2.093157049449537e-05\\
6.103515625e-05	1.193028178761486e-06\\
};

\end{axis}
\end{tikzpicture}
  \caption{Observed decay rate of the local error
    $\lVert u_1-u(\tau)\rVert_\infty$
    (see formula~\cref{eq:Strang1})
    for the function $v(x,y)=(16x(1-x)y(1-y))^p$, with $p=0,1,2,3$.
    The spatial discretization is performed with $250^2$ interior points
    and standard second-order finite differences.
    The slope of the dotted line is one, that of the dashed line is
    two and that of the solid line is three.}
  \label{fig:motivating}
\end{figure}
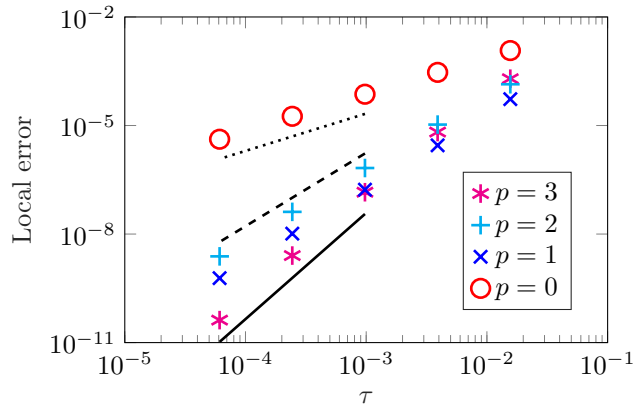

It is easy to verify numerically
that such an approximation is subject to order reduction if the
function $v$ and/or its derivatives do not satisfy
the boundary conditions.
Indeed, consider $A=\partial_{xx}$ and
$B=\partial_{yy}$ in the square domain $\Omega=(0,1)^2$ with homogeneous
Dirichlet boundary conditions, and $v(x,y) =(16x(1-x)y(1-y))^p$,
for $p=0,1,2,3$. Then, we clearly see in \cref{fig:motivating} that
the decay rate of the local error in the approximation of $u(\tau)$
is three (as expected) only for $p=3$, it drops to two for $p=2$ or $p=1$
and to one for $p=0$.
The order reduction phenomenon due to the presence of terms that do not satisfy
boundary conditions is actually well known in the literature, both in the
context of splitting~\cite{EO15b,EO16} and of exponential integrators
\cite{AMCR17,CR22}.
This very simple example motivates the investigation of
splitting approximations for the $\varphi$ functions different from
\cref{eq:Strang1}, which may have less restrictions
coming from the boundary conditions. In particular, in \cref{sec:loc}
we study the local error of the approximations using the
abstract framework of strongly continuous commuting semigroups.
The analysis is carried out for the case of two operators,
but it easily extends to more operators.
We also present numerical experiments that verify the
theoretical findings.
Then, in \cref{sec:glob}, we consider two split
versions (denoted ERK2L and ERK2) of ETD2RK,
a popular second-order exponential integrator of the Runge--Kutta type.
We show that, under appropriate assumptions, in our cases of interest ERK2L can
have order reduction down to one, while ERK2 is always second-order convergent
(in the stiff sense). We demonstrate this with some numerical examples
in two and three spatial dimensions.
Finally, we draw the conclusions and present possible further developments
in \cref{sec:concl}.
\section{Splitting of
  \texorpdfstring{$\varphi$}{phi} functions}\label{sec:loc}
As mentioned in the introduction, we perform our analysis in an abstract
framework, i.e., we consider the abstract ordinary differential
equation
\begin{equation}\label{eq:AODE}
  \left\{
  \begin{aligned}
    u'(t) &= (A+B)u(t) + g(u(t)) = f(u(t)), \\
    u(0) &= u_0,
  \end{aligned}
  \right.
\end{equation}
associated to the partial differential equation \cref{eq:PDE}.
A reader not familiar with this formalism is invited to
consult for example~\cite{EN00,L95,P83}. 
In this section we make the following assumptions.
\begin{assumption}\label{assum:scs}
  Let $X$ be a Banach space with norm $\norm{\cdot}$, and
  let
$A \colon \mathcal{D}(A)\subset X \rightarrow X$ and
$B \colon \mathcal{D}(B)\subset X \rightarrow X$ be linear operators.
We assume
that $A$ and $B$ generate strongly continuous
semigroups $\{\rme^{tA}\}_{t \geq 0}$ and $\{\rme^{tB}\}_{t \geq 0}$.
\end{assumption}
\begin{assumption}\label{assum:commut}
We assume that the operators $A$ and $B$ \emph{commute} in the sense
of the semigroups, i.e., $\rme^{tA}\rme^{tB} = \rme^{tB}\rme^{tA}$.
\end{assumption}
Under \cref{assum:scs,assum:commut}
we have that 
$\rme^{tA}\rme^{tB}$ is a strongly continuous semigroup
with generator
$\overline{A+B}\colon \mathcal{D}(A+B)=\mathcal{D}(A)\cap\mathcal{D}(B)\subset X\to X$
\cite{EN00,T59}, i.e.,
\begin{equation}\label{eq:sgp}
\rme^{tA}\rme^{tB}=\rme^{t(A+B)}.
\end{equation}
Note that in our cases of interest $A+B$ is a closed operator, and therefore
we omit the closure from the notation.
Furthermore, the following bound
\begin{equation}\label{eq:expbound}
    \norm{\rme^{tA}} \leq C, \quad  0 \leq t \leq T,
\end{equation}
holds for the operator $A$
(and analogously for
the operators $B$ and $A+B$).
Here and throughout the paper $C$ denotes a generic constant
that can have different values at different occurrences, but
it is independent of the stiffness of the problem.
Thanks to the bound \cref{eq:expbound}, we get
that the $\varphi_\ell$ functions appearing in exponential
integrators, defined by
\begin{equation}\label{eq:phidef}
  \varphi_0(z)=\rme^z,\quad 
  \varphi_\ell(z)=\frac{1}{(\ell-1)!}\int_0^1\theta^{\ell-1}\rme^{(1-\theta)z}d\theta=
  \sum_{k=0}^\infty\frac{z^k}{(\ell+k)!},\quad \ell>0
\end{equation}
satisfy $\norm{\varphi_\ell(tA)} \leq C$ for $0\leq t \leq T$.
Also, from their definition it easily
follows that the recursive formula
\begin{equation}\label{eq:phirec}
\varphi_\ell(z)=
  \frac{\varphi_{\ell-1}(z)-\varphi_{\ell-1}(0)}{z},\quad \ell>0
\end{equation}
holds true.
Finally, we note that we present here the analysis
for two operators only (for simplicity of exposition).
The generalization to more than two operators is straightforward, see \cref{rem:dop}.

\subsection{Splitting of \texorpdfstring{$\varphi_1$}{phi1}}
We start the analysis with the split approximation of $\tau\varphi_1(\tau(A+B))$
mentioned in the introduction (see formula~\cref{eq:Strang1}), i.e.,
\begin{equation}\label{eq:expphi}
  \tau\varphi_1(\tau(A+B)) \approx \tau\rme^{\frac{\tau}{2}A}\varphi_1(\tau B).
\end{equation}
Note that for \emph{bounded operators} $A$ and $B$
it follows by standard
Taylor expansion that
\begin{equation*}
  \tau\varphi_1(\tau(A+B)) -\tau\rme^{\frac{\tau}{2}A}\varphi_1(\tau B)
  =\frac{\tau^3}{24}(A^2+2BA)+\mathcal{O}(\tau^4).
  \end{equation*}
\begin{lemma}\label{lem:expphi}
Under \cref{assum:scs,assum:commut} we have
\begin{equation*}
  \begin{split}
\tau\varphi_1(\tau(A+B))
  -  \tau\rme^{\frac{\tau}{2}A}\varphi_1(\tau B) =
  &-\int_0^\tau \int_0^s\left(\frac{\tau}{2}-s\right)
  \rme^{\tau A}\rme^{(\tau-\xi)B}BA d\xi  ds\\
&-\int_0^\tau \int_s^{\frac{\tau}{2}}\int_0^\sigma 
   \rme^{(\tau-s)B}\rme^{(\tau-\xi) A}A^2d\xi d\sigma ds.
  \end{split}
\end{equation*}
\end{lemma}
\begin{proof}
  From the integral representation of the $\varphi_1$ function
  (see~\cref{eq:phidef}) we have
  \begin{equation}\label{eq:expphi1_1}
    \tau\varphi_1(\tau(A+B))-\tau\rme^{\frac{\tau}{2}A}\varphi_1(\tau B)=
    \int_0^\tau\left(\rme^{(\tau-s)(A+B)}-
    \rme^{\frac{\tau}{2}A}\rme^{(\tau-s)B}\right)ds.
  \end{equation}
By using the commutativity of the semigroups and inserting
\begin{equation}\label{eq:iph}
  \rme^{\frac{\tau}{2}A}=\rme^{(\tau-s)A}-\int_s^{\frac{\tau}{2}}
  \rme^{(\tau-\sigma)A}Ad\sigma
\end{equation}
into the right-hand side, we get
\begin{equation}\label{eq:expphi1_2}
    \tau\varphi_1(\tau(A+B))-\tau\rme^{\frac{\tau}{2}A}\varphi_1(\tau B)=
  \int_0^\tau \int_s^{\frac{\tau}{2}}
  \rme^{(\tau-s)B}\rme^{(\tau-\sigma)A}Ad\sigma ds.
\end{equation}
Now, by substituting
\begin{equation}\label{eq:ipp}
  \rme^{(\tau-\sigma)A}=\rme^{\tau A}-\int_0^{\sigma}
  \rme^{(\tau-\xi)A}Ad\xi,
\end{equation}
the right-hand side of \cref{eq:expphi1_2} becomes
\begin{equation*}
  \int_0^\tau \int_s^{\frac{\tau}{2}}
   \rme^{(\tau-s)B}\rme^{\tau A}Ad\sigma ds
  -\int_0^\tau \int_s^{\frac{\tau}{2}}\int_0^\sigma 
   \rme^{(\tau-s)B}\rme^{(\tau-\xi) A}A^2d\xi d\sigma ds.
\end{equation*}
The second term is in the form given in the statement.
Concerning the first term, integrating with respect to~$\sigma$ we have
\begin{equation*}
\int_0^\tau\!\! \int_s^{\frac{\tau}{2}}\!
 \rme^{(\tau-s)B}\rme^{\tau A}Ad\sigma ds=
\int_0^\tau\!\! \left(\frac{\tau}{2}-s\right)
\rme^{(\tau-s)B}\rme^{\tau A}A ds=
\int_0^\tau\!\! \left(\frac{\tau}{2}-s\right)
\rme^{\tau A}\rme^{(\tau-s)B}A ds.
\end{equation*}
Let us consider now \cref{eq:ipp} with $s$ and $B$ instead of $\sigma$
and $A$, respectively. If we insert it
%\begin{equation*}
%  \rme^{(\tau-s)B}=\rme^{\tau B}-\int_0^{s}
%  \rme^{(\tau-\xi)B}Bd\xi
%\end{equation*}
in the right-hand side above,
we get
\begin{equation*}
  \int_0^\tau \left(\frac{\tau}{2}-s\right)
\rme^{\tau A}\rme^{(\tau-s)B}A ds
=-\int_0^\tau \int_0^s\left(\frac{\tau}{2}-s\right)
\rme^{\tau A}\rme^{(\tau-\xi)B}BA d\xi  ds,
\end{equation*}
where we used the fact that
\begin{equation*}
  \int_0^\tau\left(\frac{\tau}{2}-s\right)ds=0.
\end{equation*}
This concludes the proof.
\end{proof}

\begin{proposition}\label{prop:expphi}
  Let \cref{assum:scs,assum:commut}
  be valid. Then, the following
bounds on the local error
\begin{equation*}
\lambda_1=
  \norm{\tau\varphi_1(\tau(A+B))v-\tau\rme^{\frac{\tau}{2}A}\varphi_1(\tau B)v}
\end{equation*}
hold:
\begin{enumerate}[label=(\alph*)]
%\item if $v\in\mathcal{D}(A^2)\cap\mathcal{D}(B)$, then
%$\lambda_1\leq C\tau^3$;
%\item if $v\in\mathcal{D}(A)\setminus\mathcal{D}(A^2)$ or
%  $v\in\mathcal{D}(A)\setminus\mathcal{D}(B)$, then
%$\lambda_1\leq C\tau^2$;
%\item if $v\not\in\mathcal{D}(A)$, then
%  $\lambda_1\leq C\tau$.
\item if $v\in\mathcal{D}(A^2)\cap\mathcal{D}(BA)$, then
$\lambda_1\leq C\tau^3$;
\item if $v\in\mathcal{D}(A)$, then
$\lambda_1\leq C\tau^2$;
\item if $v\not\in\mathcal{D}(A)$, then
  $\lambda_1\leq C\tau$.
\end{enumerate}
\end{proposition}
\begin{proof}
  The first statement follows directly
  from \cref{lem:expphi} and the bounds on the
  semigroups generated by
  $A$ and $B$, see~\cref{eq:expbound}.
  For the second statement, we note that the argument below
  \cref{eq:expphi1_2}
  cannot be made, and therefore we
  obtain the result by bounding \cref{eq:expphi1_2} using the
  estimate~\cref{eq:expbound}
  on the semigroups.
  For the third statement, similarly to the previous case, we stop
  at \cref{eq:expphi1_1} and we obtain the desired result.
\end{proof}
Note that, in our cases of interest, the order reduction occurs because the
boundary conditions are not satisfied, and \emph{not} because of the lack of
spatial regularity
of the functions. This means, for example, that
a sufficiently smooth function $v$ satisfies statement (a) in~\cref{prop:expphi} if
$v$ satisfies the boundary conditions of $A$ and $B$, and in addition $Av$
satisfies the boundary conditions of $A$
(see \cref{sec:locerr} for some numerical examples).

We observe that the approximation \cref{eq:expphi} is not symmetric with
respect to the operators $A$ and $B$. In fact, if we interchange
them, we obtain a different approximation with results
(a)--(c) in \cref{prop:expphi} modified accordingly.
A symmetric approximation is
\begin{equation}\label{eq:expexp}
  \tau\varphi_1(\tau(A+B))
  \approx \tau\rme^{\frac{\tau}{2}A}\rme^{\frac{\tau}{2} B}
  =\tau\rme^{\frac{\tau}{2}(A+B)}.
\end{equation}
Again, in the case of bounded operators $A$ and $B$, we have by Taylor
expansion that
\begin{equation*}
  \tau\varphi_1(\tau(A+B)) - \tau\rme^{\frac{\tau}{2}(A+B)}
  =\frac{\tau^3}{24}(A+B)^2+\mathcal{O}(\tau^4).
\end{equation*}
\begin{lemma}\label{lem:expexp}
  Under \cref{assum:scs,assum:commut} we have
\begin{multline*}
    \tau\varphi_1(\tau(A+B)) - \tau\rme^{\frac{\tau}{2}(A+B)} =
    \int_0^\tau\int_0^s\int_{\frac{\tau}{2}}^\sigma \rme^{(\tau-\xi)(A+B)}(A+B)^2d\xi d\sigma ds\\
    +\frac{\tau^3}{4}\rme^{\frac{\tau}{2}(A+B)}\varphi_2\left(\tfrac{\tau}{2}(A+B)\right)(A+B)^2.
\end{multline*}
\end{lemma}
\begin{proof}
Similarly to the proof of \cref{lem:expphi}, from the integral representation
of the $\varphi_1$ function \cref{eq:phidef} we have
\begin{equation}\label{eq:expexp_1}
  \tau\varphi_1(\tau(A+B)) - \tau\rme^{\frac{\tau}{2}(A+B)}
  = \int_0^\tau\left(\rme^{(\tau-s)(A+B)}-
  \rme^{\frac{\tau}{2}(A+B)}\right)ds.
\end{equation}
Then, by exploiting \cref{eq:ipp} with $s$ and $A+B$ instead of $\sigma$
and $A$, respectively, the right-hand side becomes
\begin{equation*}
  \int_0^\tau\left(\rme^{\tau(A+B)}-
  \int_0^s\rme^{(\tau-\sigma)(A+B)}(A+B)d\sigma
  -\rme^{\frac{\tau}{2}(A+B)}\right)ds.
\end{equation*}
We equivalently write this expression as
\begin{equation}\label{eq:expexp_2}
  -\int_0^\tau\int_0^s\rme^{(\tau-\sigma)(A+B)}(A+B)d\sigma ds
  +\frac{\tau^2}{2}\rme^{\frac{\tau}{2}(A+B)}\varphi_1\left(\tfrac{\tau}{2}(A+B)\right)(A+B),
\end{equation}
where we employed the recursive formula for the $\varphi_1$ function~\cref{eq:phirec}.
Inserting  \cref{eq:iph}
with $\sigma$ and $A+B$ instead of $s$ and $A$, respectively,
into the integral term
we get
\begin{equation*}
    \int_0^\tau\int_0^s\int_{\frac{\tau}{2}}^\sigma \rme^{(\tau-\xi)(A+B)}(A+B)^2d\xi d\sigma ds
    +\frac{\tau^3}{4}\rme^{\frac{\tau}{2}(A+B)}\varphi_2\left(\tfrac{\tau}{2}(A+B)\right)(A+B)^2,
\end{equation*}
where we exploited the recursive formula for the $\varphi_2$
function (see again~\cref{eq:phirec}) . This concludes the proof.
\end{proof}
\begin{proposition}\label{prop:expexp}
  Let \cref{assum:scs,assum:commut}
  be valid. Then, the following
bounds on the local error
\begin{equation*}
\lambda_2=
  \norm{\tau\varphi_1(\tau(A+B))v-\tau\rme^{\frac{\tau}{2}(A+B)}v}
\end{equation*}
hold:
\begin{enumerate}[label=(\alph*)]
\item if $v\in\mathcal{D}((A+B)^2)$, then
$\lambda_2\leq C\tau^3$;
\item if $v\in\mathcal{D}(A+B)$, then
$\lambda_2\leq C\tau^2$;
\item if $v\not\in\mathcal{D}(A+B)$, then
  $\lambda_2\leq C\tau$.
\end{enumerate}
\end{proposition}
\begin{proof}
Using arguments similar to the proof of \cref{prop:expphi},
exploiting the bounds on the semigroup and on the $\varphi_2$ function,
the first statement follows directly from \cref{lem:expexp}, while the
second and the
third statement from \cref{eq:expexp_2} and~\cref{eq:expexp_1},
respectively.
\end{proof}

We finally consider the symmetric
approximation 
\begin{equation}\label{eq:phiphi}
  \tau\varphi_1(\tau(A+B)) \approx \tau\varphi_1(\tau A)\varphi_1(\tau B).
\end{equation}
This is actually the third-order approximation already
employed in~\cite{CC24bis}.
By Taylor expansion, in case of commuting bounded operators, we have
\begin{equation*}
  \tau\varphi_1(\tau(A+B)) - \tau\varphi_1(\tau A)\varphi_1(\tau B)
  =\frac{\tau^3}{12}BA+\mathcal{O}(\tau^4).
  \end{equation*}
We can then prove the following.
\begin{lemma}\label{lem:phi1phi1}
Under \cref{assum:scs,assum:commut} we have
\begin{equation*}
  \tau\varphi_1(\tau (A+B)) -
  \tau\varphi_1(\tau A)\varphi_1(\tau B) =
  \frac{1}{\tau}\int_0^\tau\int_0^\tau \int_s^\sigma
  (\tau-s)\rme^{(\tau-\xi)A}\varphi_1((\tau-s)B)
  BAd\xi d\sigma ds.
\end{equation*}
\end{lemma}
\begin{proof}
From the integral definition of the $\varphi_1$ function \cref{eq:phidef},
and using the commutativity of the semigroups, we have
\begin{multline}\label{eq:phi1phi1_1}
  \tau\varphi_1(\tau (A+B)) - \tau\varphi_1(\tau A)\varphi_1(\tau B) =
  \int_0^\tau \rme^{(\tau-s)(A+B)}ds \\
  -\frac{1}{\tau}\int_0^\tau \rme^{(\tau-s)B}
  \left(\int_0^\tau \rme^{(\tau-\sigma)A}d\sigma\right)ds.
\end{multline}
Now, by using the fact that
\begin{equation*}
  \rme^{(\tau-\sigma)A}=\rme^{(\tau-s)A}-
  \int_s^\sigma\rme^{(\tau-\xi)A}Ad\xi,
\end{equation*}
the right-hand side becomes
\begin{equation*}
  \int_0^\tau \rme^{(\tau-s)(A+B)}ds
  -\int_0^\tau \rme^{(\tau-s)B}
  \left(
  \rme^{(\tau-s)A}-
  \frac{1}{\tau}\int_0^\tau\int_s^\sigma\rme^{(\tau-\xi)A}Ad\xi d\sigma\right)ds.
\end{equation*}
Exploiting \cref{eq:sgp}, we then get
\begin{equation}\label{eq:phi1phi1_2}
 \tau\varphi_1(\tau(A+B))-\tau\varphi_1(\tau A)\varphi_1(\tau B) =
  \frac{1}{\tau}\int_0^\tau\int_0^\tau \int_s^\sigma\rme^{(\tau-s)B}\rme^{(\tau-\xi)A}Ad\xi d\sigma ds.
\end{equation}
Finally, by using the commutativity of the semigroups, the right-hand side can
be equivalently written as
\begin{multline*}
\frac{1}{\tau}\int_0^\tau\int_0^\tau \int_s^\sigma
\rme^{(\tau-\xi)A}
\rme^{(\tau-s)B}Ad\xi d\sigma ds=
\frac{1}{\tau}\int_0^\tau\int_0^\tau \int_s^\sigma
\rme^{(\tau-\xi)A}(\rme^{(\tau-s)B}-I)Ad\xi d\sigma ds\\
=\frac{1}{\tau}\int_0^\tau\int_0^\tau \int_s^\sigma
(\tau-s)
\rme^{(\tau-\xi)A}\varphi_1((\tau-s)B)BAd\xi d\sigma ds.
\end{multline*}
Here, we also exploited that
\begin{equation*}
  \int_0^\tau\int_0^\tau \int_s^\sigma \rme^{(\tau-\xi)A}Ad\xi d\sigma ds=
  \int_0^\tau\int_0^\tau \left(\rme^{(\tau-s)A}-\rme^{(\tau-\sigma)A}\right)
  d\sigma ds=
  0
\end{equation*}
and the recursive definition of the $\varphi_1$ function~\cref{eq:phirec}.
This concludes the proof.
\end{proof}
\begin{proposition}\label{prop:phi1phi1}
  Let \cref{assum:scs,assum:commut}
  be valid. Then, the following
bounds on the local error
\begin{equation*}
\lambda_3=
  \norm{\tau\varphi_1(\tau(A+B))v-\tau\varphi_1(\tau A)\varphi_1(\tau B)v}
\end{equation*}
hold:
\begin{enumerate}[label=(\alph*)]
\item if $v\in\mathcal{D}(BA)$,
  then $\lambda_3 \leq C\tau^3$;
\item if $v\in\mathcal{D}(A)$,
  then $\lambda_3 \leq C\tau^2$;
\item if $v\not\in\mathcal{D}(A)$,
  then $\lambda_3 \leq C\tau$.
\end{enumerate}
\end{proposition}
\begin{proof}
The first statement follows directly from \cref{lem:phi1phi1}, using the 
bounds on the semigroup and on the $\varphi_1$ function.
As in the previous propositions, the second and the third
statement follow from intermediate results,
in this case \cref{eq:phi1phi1_2} and \cref{eq:phi1phi1_1},
respectively.
\end{proof}
Note that in \cref{eq:phiphi} we can interchange arbitrarily the
operators $A$ and $B$, since it is a symmetric approximation. Therefore,
statements similar to (a)--(c) in \cref{prop:phi1phi1} in terms of
$\mathcal{D}(AB)$ and $\mathcal{D}(B)$ hold true.
\subsection{Splitting of \texorpdfstring{$\varphi_\ell$}{phiell}}
We consider here
 the general symmetric
approximation
\begin{equation}\label{eq:phiellphiell}
  \tau\varphi_\ell(\tau(A+B)) \approx
  \ell!\tau\varphi_\ell(\tau A)\varphi_\ell(\tau B),
\end{equation}
which was already employed in~\cite{CC24bis}.
\begin{lemma}\label{lem:phiellphiell}
Under \cref{assum:scs,assum:commut} we have
\begin{multline*}
  \tau\varphi_\ell(\tau (A+B)) -
  \ell!\tau\varphi_\ell(\tau A)\varphi_\ell(\tau B) = \\
  \frac{\ell}{\tau^{2\ell-1}(\ell-1)!}\int_0^\tau\int_0^\tau\int_s^\sigma
  (\sigma s)^{(\ell-1)}(\tau-s)
  \rme^{(\tau-\xi)A}\varphi_1((\tau-s)B)BA d\xi d\sigma ds.
\end{multline*}
\end{lemma}
\begin{proof}
From the integral definition of the $\varphi_\ell$ function \cref{eq:phidef},
using the commutativity of the semigroups, we have
\begin{multline*}
    \tau\varphi_\ell(\tau (A+B)) -
    \ell!\tau\varphi_\ell(\tau A)\varphi_\ell(\tau B) =
  \frac{1}{\tau^{\ell-1}(\ell-1)!}\int_0^\tau s^{\ell-1}\rme^{(\tau-s)(A+B)}ds \\
  -\frac{\ell}{\tau^{2\ell-1}(\ell-1)!}\int_0^\tau
  s^{\ell-1}\rme^{(\tau-s)B}
  \left(\int_0^\tau \sigma^{\ell-1}\rme^{(\tau-\sigma)A} d\sigma\right)ds.
\end{multline*}
Taking into account that
\begin{equation*}
  \int_0^\tau\int_0^\tau(\sigma s)^{\ell-1}\left(\rme^{(\tau-s)A}-\rme^{(\tau-\sigma)A}\right) d\sigma ds = 0,
\end{equation*}
the proof follows along the steps of \cref{lem:phi1phi1}.
\end{proof}
For the split exponential integrators that we will consider in
\cref{sec:glob} we need local error bounds
for the $\varphi_2$ function. Therefore,
we state the following proposition just for the case $\ell=2$.
\begin{proposition}\label{prop:phi2phi2}
  Let \cref{assum:scs,assum:commut}
  be valid. Then, the following
bounds on the local error
\begin{equation*}
\lambda_4=
  \norm{\tau\varphi_2(\tau(A+B))v-2\tau\varphi_2(\tau A)\varphi_2(\tau B)v}
\end{equation*}
hold:
\begin{enumerate}[label=(\alph*)]
\item if $v\in\mathcal{D}(BA)$,
  then $\lambda_4 \leq C\tau^3$;
\item if $v\in\mathcal{D}(A)$, then
  $\lambda_4 \leq C\tau^2$;
\item if $v\not\in\mathcal{D}(A)$,
  then $\lambda_4 \leq C\tau$.
\end{enumerate}
\end{proposition}
\begin{proof}
  The result follows from \cref{lem:phiellphiell} and reasoning
  similar to that in \cref{prop:phi1phi1}.
\end{proof}

\begin{remark}\label{rem:dop}
  As mentioned at the beginning of the section, we fully developed the
  analysis for
  the case of two commuting operators. The case with $d$ operators follows
straightforwardly. In fact, we have
(see also~\cite{CC24bis}),
\begin{equation*}
  \lVert\tau \varphi_\ell(\tau(A_1+\cdots+A_d))v -
  (\ell!)^{d-1}\tau\varphi_\ell(\tau A_1)\cdots\varphi_\ell(\tau A_d)v\rVert
  \le C\tau^3,
\end{equation*}
if $v$ is a sufficiently smooth function that satisfies the
boundary conditions of the operators $A_1,\ldots,A_d$.
If not, we encounter similar order reductions as for the case of two operators.
\end{remark}
\subsection{Numerical experiments on the split
  approximations}\label{sec:locerr}
In the following experiments we compare the split approximations to
$\tau\varphi_1(\tau(A+B))v$ and
$\tau\varphi_2(\tau(A+B))v$ introduced
in the previous sections. To this aim, we consider
the space of continuous functions $X=C(\Omega)$, equipped with the infinity norm,
on the square domain $\Omega=(0,1)^2$.
The operators are $A=\partial_{xx}$ and $B=\partial_{yy}$ with homogeneous Dirichlet
boundary conditions. We consider smooth functions
$v=v(x,y)$ such that their partial derivatives satisfy or not the boundary conditions.
The operators are approximated by
standard second-order
finite differences (with $250^2$ interior discretization points).
We consider a common range of time step sizes $\tau$ and measure the
decay rate of the local error for the different approximations.
The selected functions $v$ are presented in \cref{tab:domains}.
In \cref{tab:rates} we collect the expected decay rates of the
local errors, according to the results obtained in the previous
sections.
The numerical results, summarized in \cref{fig:localerrors}, match satisfactorily
the theoretical findings.

\begin{table}[!ht]
  \renewcommand{\arraystretch}{1.1}
  \caption{Summary of the functions $v=v(x,y)$ and of the split
    approximations to $\tau\varphi_1(\tau(A+B))v$ and
    $\tau\varphi_2(\tau(A+B))v$,
    see \cref{eq:expphi,eq:expexp,eq:phiphi,eq:phiellphiell},
    employed in the numerical experiments of
    \cref{sec:locerr}.
    Here $A=\partial_{xx}$, $B=\partial_{yy}$
    (with homogeneous Dirichlet boundary conditions) and $X=C((0,1)^2)$.}
  \label{tab:loc}
  \addtocounter{table}{-1}
  \centering

  \subfloat[Functions $v=v(x,y)$.\label{tab:domains}]{
        \begin{tabular}{l|c|c|c|c|c}
      $v=v(x,y)$ & $\mathcal{D}(A)$ & $\mathcal{D}(A^2)$ &
      $\mathcal{D}(BA)$ & $\mathcal{D}(A+B)$ & $\mathcal{D}((A+B)^2)$\\
    \hline
    $v_1=4^6(x(1-x)y(1-y))^3$ & \checkmark & \checkmark & \checkmark  & \checkmark & \checkmark\\
    $v_2=4^4(x(1-x))^3y(1-y)$ & \checkmark & \checkmark & \checkmark  & \checkmark & $\times$\\
    $v_3=4^2x(1-x)y(1-y)$ & \checkmark & $\times$ & \checkmark  & \checkmark & $\times$\\
    $v_4=4x(1-x)$ & \checkmark & $\times$ & $\times$  & $\times$ & $\times$\\
    $v_5=1$ & $\times$ & $\times$ & $\times$  & $\times$ & $\times$\\
\end{tabular}}
 
\subfloat[Expected decay rates of the local errors.\label{tab:rates}]{
  \begin{tabular}{c|c|c|c||c}
    $v$ & $\tau\rme^{\frac{\tau}{2} A}\varphi_1(\tau B)v$ &
    $\tau\rme^{\frac{\tau}{2} A}\rme^{\frac{\tau}{2} B}v$ &
    $\tau\varphi_1(\tau A)\varphi_1(\tau B)v$ &
    $2\tau\varphi_2(\tau A)\varphi_2(\tau B)v$\\
    \hline
    $v_1$ & 3 & 3 & 3 & 3\\
    $v_2$ & 3 & 2 & 3 & 3\\
    $v_3$ & 2 & 2 & 3 & 3\\
    $v_4$ & 2 & 1 & 2 & 2\\
    $v_5$ & 1 & 1 & 1 & 1
  \end{tabular}}
\stepcounter{table}
\end{table}
\begin{figure}[!htbp]
  \centering
\begin{tikzpicture}

\begin{axis}[%
width=1.7in,
height=1.3in,
at={(0.769in,0.477in)},
scale only axis,
xmode=log,
xmin=1e-5,
xmax=1e-1,
xminorticks=true,
xlabel style={font=\color{white!15!black}},
xlabel={$\tau$},
ymode=log,
ymin=1e-13,
ymax=1e-2,
yminorticks=true,
ylabel style={font=\color{white!15!black}},
ylabel={Local error for $v_1$},
xtick = {1e-5,1e-3,1e-1},
xticklabels = {$10^{-5}$,$10^{-3}$,$10^{-1}$},
axis background/.style={fill=white},
legend style={at={(0.03,0.65)}, anchor=south west, legend cell align=left,
              align=left, draw=white!15!black, nodes={scale=0.65, transform shape}}
]
%\addlegendentry{Expexp}

%\addplot [color=black, forget plot]
%  table[row sep=crcr]{%
%0.00025	2.81282122839695e-09\\
%0.000125	3.51602653549619e-10\\
%};
\addplot [color=red, only marks,line width=1pt, mark size = 3.5pt, mark=o, mark options={solid, red}]
  table[row sep=crcr]{%
0.015625	0.000196800256324108\\
0.00390625	6.59571613521136e-06\\
0.000976562499999999	1.43283760789678e-07\\
0.000244140625	2.5440361607379e-09\\
6.10351562499999e-05	4.1805181484421e-11\\
};
\addlegendentry{$\tau\rme^{\frac{\tau}{2}A}\varphi_1(\tau B)v$}
%\addlegendentry{Expphi}

\addplot [color=blue,  line width=1pt, mark size = 3.5pt,only marks, mark=x, mark options={solid, blue}]
  table[row sep=crcr]{%
0.015625	0.000303339664783273\\
0.00390625	7.4128458303511e-06\\
0.000976562499999999	1.42360434096912e-07\\
0.000244140625	2.5405583927564e-09\\
6.10351562499999e-05	4.17937110186845e-11\\
};
\addlegendentry{$\tau\rme^{\frac{\tau}{2}A}\rme^{\frac{\tau}{2}B}v$}

%\addplot [color=black,line width=1pt, forget plot]
%  table[row sep=crcr]{%
%0.015625	0.000701374559683331\\
%6.10351562499999e-05	4.1805181484421e-11\\
%};

\addplot [color=magenta, only marks,line width=1pt, mark size = 3.5pt, mark=+, mark options={solid, magenta}]
  table[row sep=crcr]{%
0.015625	8.92409740465576e-05\\
0.00390625	2.37325827525877e-06\\
0.000976562499999999	4.26446397091369e-08\\
0.000244140625	6.90136975426682e-10\\
6.10351562499999e-05	1.08785573865039e-11\\
};
\addlegendentry{$\tau\varphi_1(\tau A)\varphi_1(\tau B)v$}
%\addlegendentry{Phiphi}

%\addplot [color=black, line width=1pt,forget plot]
%  table[row sep=crcr]{%
%0.015625	0.000182511907041771\\
%6.10351562499999e-05	1.08785573865039e-11\\
%};
\end{axis}

\begin{axis}[%
width=1.7in,
height=1.3in,
at={(0.769in,0.477in)},
scale only axis,
xmode=log,
xmin=1e-5,
xmax=1e-1,
xminorticks=false,
xmajorticks=false,
ymode=log,
ymin=1e-13,
ymax=1e-2,
yminorticks=false,
xtick = {},
xticklabels = {},
ytick = {},
yticklabels = {},
legend style={at={(0.4,0.06)}, anchor=south west, legend cell align=left,
              align=left, draw=white!15!black,nodes={scale=0.65, transform shape}}
]
\addplot [color=green, only marks, line width=1pt, mark size = 3.5pt, mark=asterisk, mark options={solid, green}]
  table[row sep=crcr]{%
0.015625	3.43053580106871e-05\\
0.00390625	8.21346022527928e-07\\
0.000976562499999999	1.43488956629608e-08\\
0.000244140625	2.30585596393252e-10\\
6.10351562499999e-05	3.62831124059394e-12\\
};
\addlegendentry{$2\tau\varphi_2(\tau A)\varphi_2(\tau B)v$}
%\addlegendentry{Phiphi (phi2)}

\addplot [color=black, line width=1pt, forget plot]
  table[row sep=crcr]{%
0.015625	6.08729613986724e-05\\
6.10351562499999e-05	3.62831124059394e-12\\
};
\end{axis}

\end{tikzpicture}%
\begin{tikzpicture}

\begin{axis}[%
width=1.7in,
height=1.3in,
at={(0.769in,0.477in)},
scale only axis,
xmode=log,
xmin=1e-5,
xmax=1e-1,
xminorticks=true,
xlabel style={font=\color{white!15!black}},
xlabel={$\tau$},
ymode=log,
ymin=1e-13,
ymax=1e-2,
yminorticks=true,
ylabel style={font=\color{white!15!black}},
ylabel={Local error for $v_2$},
xtick = {1e-5,1e-3,1e-1},
xticklabels = {$10^{-5}$,$10^{-3}$,$10^{-1}$},
axis background/.style={fill=white},
legend style={at={(0.17,0.158)}, anchor=south west, legend cell align=left, align=left, draw=white!15!black}
]
\addplot [color=blue,  line width=1pt, mark size = 3.5pt,only marks, mark=x, mark options={solid, blue}]
  table[row sep=crcr]{%
0.015625	0.000166339484640676\\
0.00390625	5.76413240426996e-06\\
0.000976562499999999	1.78223691809891e-07\\
0.000244140625	1.03746622820659e-08\\
6.10351562499999e-05	6.06278494013069e-10\\
};
%\addlegendentry{Expexp}

\addplot [color=black, dashed, line width = 1pt, forget plot]
  table[row sep=crcr]{%
0.015625	3.97330673836404e-05\\
6.10351562499999e-05	6.06278494013069e-10\\
};
\addplot [color=red, only marks,line width=1pt, mark size = 3.5pt, mark=o, mark options={solid, red}]
  table[row sep=crcr]{%
0.015625	0.000166411206988338\\
0.00390625	5.76413240508537e-06\\
0.000976562499999999	1.27031903509152e-07\\
0.000244140625	2.33786427504143e-09\\
6.10351562499999e-05	3.9554033175186e-11\\
};
%\addlegendentry{Expphi}

%\addplot [color=black,line width=1pt, forget plot]
%  table[row sep=crcr]{%
%0.015625	0.000663606558251261\\
%6.10351562499999e-05	3.9554033175186e-11\\
%};

  \addplot [color=magenta, only marks,line width=1pt, mark size = 3.5pt, mark=+, mark options={solid, magenta}]
  table[row sep=crcr]{%
0.015625	4.23763124930785e-05\\
0.00390625	8.68323490377935e-07\\
0.000976562499999999	1.45537075520348e-08\\
0.000244140625	2.31433698512926e-10\\
6.10351562499999e-05	3.63207747603216e-12\\
};
%\addlegendentry{Phiphi}

%\addplot [color=black, line width=1pt,forget plot]
%  table[row sep=crcr]{%
%0.015625	6.09361483441262e-05\\
%6.10351562499999e-05	3.63207747603216e-12\\
%};

\addplot [color=green, only marks, line width=1pt, mark size = 3.5pt, mark=asterisk, mark options={solid, green}]
  table[row sep=crcr]{%
0.015625	1.51856192998822e-05\\
0.00390625	2.94934862929332e-07\\
0.000976562499999999	4.87405985352945e-09\\
0.000244140625	7.72350544793725e-11\\
6.10351562499999e-05	1.21104758572755e-12\\
};
%\addlegendentry{Phiphi (phi2)}

\addplot [color=black, line width=1pt, forget plot]
  table[row sep=crcr]{%
0.015625	2.03180069320296e-05\\
6.10351562499999e-05	1.21104758572755e-12\\
};
\end{axis}

\end{tikzpicture}%
\\
\begin{tikzpicture}

\begin{axis}[%
width=1.7in,
height=1.3in,
at={(0.769in,0.477in)},
scale only axis,
xmode=log,
xmin=1e-5,
xmax=1e-1,
xminorticks=true,
xlabel style={font=\color{white!15!black}},
xlabel={$\tau$},
ymode=log,
ymin=1e-13,
ymax=1e-2,
yminorticks=true,
ylabel style={font=\color{white!15!black}},
ylabel={Local error for $v_3$},
xtick = {1e-5,1e-3,1e-1},
xticklabels = {$10^{-5}$,$10^{-3}$,$10^{-1}$},
axis background/.style={fill=white},
legend style={at={(0.17,0.158)}, anchor=south west, legend cell align=left, align=left, draw=white!15!black}
]
\addplot [color=blue,  line width=1pt, mark size = 3.5pt,only marks, mark=x, mark options={solid, blue}]
  table[row sep=crcr]{%
0.015625	5.36563117408319e-05\\
0.00390625	2.82727222706821e-06\\
0.000976562499999999	1.68582237082646e-07\\
0.000244140625	1.02720121362074e-08\\
6.10351562499999e-05	6.04527877616397e-10\\
};
%\addlegendentry{Expexp}

\addplot [color=black, dashed, line width = 1pt, forget plot]
  table[row sep=crcr]{%
0.015625	3.96183389874681e-05\\
6.10351562499999e-05	6.04527877616397e-10\\
};
\addplot [color=red, only marks,line width=1pt, mark size = 3.5pt, mark=o, mark options={solid, red}]
  table[row sep=crcr]{%
0.015625	5.35166573932117e-05\\
0.00390625	2.82727222486683e-06\\
0.000976562499999999	1.68582237085343e-07\\
0.000244140625	1.02720121362565e-08\\
6.10351562499999e-05	6.04527877616397e-10\\
};
%\addlegendentry{Expphi}

%\addplot [color=black, dashed, line width=1pt, forget plot]
%  table[row sep=crcr]{%
%0.015625	3.96183389874681e-05\\
%6.10351562499999e-05	6.04527877616397e-10\\
%};
\addplot [color=magenta, only marks,line width=1pt, mark size = 3.5pt, mark=+, mark options={solid, magenta}]
  table[row sep=crcr]{%
0.015625	2.03068735090024e-05\\
0.00390625	3.17891439163049e-07\\
0.000976562499999999	4.96705376828999e-09\\
0.000244140625	7.76102170268848e-11\\
6.10351562499999e-05	1.21265986381474e-12\\
};
%\addlegendentry{Phiphi}

%\addplot [color=black, line width=1pt,forget plot]
%  table[row sep=crcr]{%
%0.015625	2.03450564697504e-05\\
%6.10351562499999e-05	1.21265986381474e-12\\
%};

\addplot [color=green, only marks, line width=1pt, mark size = 3.5pt, mark=asterisk, mark options={solid, green}]
  table[row sep=crcr]{%
0.015625	6.77658961321069e-06\\
0.00390625	1.05963813033533e-07\\
0.000976562499999999	1.65568458688212e-09\\
0.000244140625	2.58700721299721e-11\\
6.10351562499999e-05	4.04220081377511e-13\\
};
%\addlegendentry{Phiphi (phi2)}

\addplot [color=black, line width=1pt, forget plot]
  table[row sep=crcr]{%
0.015625	6.78168761680808e-06\\
6.10351562499999e-05	4.04220081377511e-13\\
};
\end{axis}

\end{tikzpicture}%
\begin{tikzpicture}

\begin{axis}[%
width=1.7in,
height=1.3in,
at={(0.769in,0.477in)},
scale only axis,
xmode=log,
xmin=1e-5,
xmax=1e-1,
xminorticks=true,
xlabel style={font=\color{white!15!black}},
xlabel={$\tau$},
ymode=log,
ymin=1e-13,
ymax=1e-2,
yminorticks=true,
ylabel style={font=\color{white!15!black}},
ylabel={Local error for $v_4$},
xtick = {1e-5,1e-3,1e-1},
xticklabels = {$10^{-5}$,$10^{-3}$,$10^{-1}$},
axis background/.style={fill=white},
legend style={at={(0.17,0.158)}, anchor=south west, legend cell align=left, align=left, draw=white!15!black}
]
\addplot [color=blue,  line width=1pt, mark size = 3.5pt,only marks, mark=x, mark options={solid, blue}]
  table[row sep=crcr]{%
0.015625	0.00108576550990061\\
0.00390625	0.000266346320842993\\
0.000976562499999999	6.60226552357389e-05\\
0.000244140625	1.62202203495621e-05\\
6.103515625e-05	3.78308905479757e-06\\
};
%\addlegendentry{Expexp}

\addplot [color=black, dotted, line width = 1pt, forget plot]
  table[row sep=crcr]{%
0.015625	0.000968470798028176\\
6.103515625e-05	3.78308905479757e-06\\
};
\addplot [color=red, only marks,line width=1pt, mark size = 3.5pt, mark=o, mark options={solid, red}]
  table[row sep=crcr]{%
0.015625	0.000101108334875885\\
0.00390625	6.31168109558795e-06\\
0.000976562499999999	3.93085180045926e-07\\
0.000244140625	2.43077565191487e-08\\
6.10351562499999e-05	1.37554021995233e-09\\
};
%\addlegendentry{Expphi}

%\addplot [color=black, dashed, line width=1pt, forget plot]
%  table[row sep=crcr]{%
%0.015625	9.01474038547956e-05\\
%6.10351562499999e-05	1.37554021995233e-09\\
%};
\addplot [color=magenta, only marks,line width=1pt, mark size = 3.5pt, mark=+, mark options={solid, magenta}]
  table[row sep=crcr]{%
0.015625	9.48186809053232e-05\\
0.00390625	5.92368468297146e-06\\
0.000976562499999999	3.69371922729223e-07\\
0.000244140625	2.25972849365646e-08\\
6.10351562499999e-05	1.33289155993143e-09\\
};
%\addlegendentry{Phiphi}

%\addplot [color=black, dashed, line width=1pt,forget plot]
%  table[row sep=crcr]{%
%0.015625	8.73523812716658e-05\\
%6.10351562499999e-05	1.33289155993143e-09\\
%};

\addplot [color=green, only marks, line width=1pt, mark size = 3.5pt, mark=asterisk, mark options={solid, green}]
  table[row sep=crcr]{%
0.015625	4.14550325468246e-05\\
0.00390625	2.58784042375518e-06\\
0.000976562499999999	1.60790287457083e-07\\
0.000244140625	9.91689727669716e-09\\
6.10351562499999e-05	5.89536471906398e-10\\
};
%\addlegendentry{Phiphi (phi2)}

\addplot [color=black, dashed, line width=1pt, forget plot]
  table[row sep=crcr]{%
0.015625	3.86358622228578e-05\\
6.10351562499999e-05	5.89536471906398e-10\\
};
\end{axis}

\end{tikzpicture}%
\\
\begin{tikzpicture}

\begin{axis}[%
width=1.7in,
height=1.3in,
at={(0.769in,0.477in)},
scale only axis,
xmode=log,
xmin=1e-5,
xmax=1e-1,
xminorticks=true,
xlabel style={font=\color{white!15!black}},
xlabel={$\tau$},
ymode=log,
ymin=1e-13,
ymax=1e-2,
yminorticks=true,
ylabel style={font=\color{white!15!black}},
ylabel={Local error for $v_5$},
xtick = {1e-5,1e-3,1e-1},
xticklabels = {$10^{-5}$,$10^{-3}$,$10^{-1}$},
axis background/.style={fill=white},
legend style={at={(0.17,0.158)}, anchor=south west, legend cell align=left, align=left, draw=white!15!black}
]
\addplot [color=blue,  line width=1pt, mark size = 3.5pt,only marks, mark=x, mark options={solid, blue}]
  table[row sep=crcr]{%
0.015625	0.00153837129212697\\
0.00390625	0.000383410978494658\\
0.000976562499999999	9.52078757190874e-05\\
0.000244140625	2.28835351376397e-05\\
6.103515625e-05	4.98867680392262e-06\\
};
%\addlegendentry{Expexp}

%\addplot [color=black, dotted, line width = 1pt, forget plot]
%  table[row sep=crcr]{%
%0.015625	0.00127710126180419\\
%6.103515625e-05	4.98867680392262e-06\\
%};
\addplot [color=red, only marks,line width=1pt, mark size = 3.5pt, mark=o, mark options={solid, red}]
  table[row sep=crcr]{%
0.015625	0.00117872951576602\\
0.00390625	0.000294330494825977\\
0.000976562499999999	7.32604967307338e-05\\
0.000244140625	1.7962482105881e-05\\
6.103515625e-05	4.1755986256652e-06\\
};
%\addlegendentry{Expphi}

%\addplot [color=black, dotted, line width = 1pt, dotted, forget plot]
%  table[row sep=crcr]{%
%0.015625	0.00106895324817029\\
%6.103515625e-05	4.1755986256652e-06\\
%};
\addplot [color=magenta, only marks,line width=1pt, mark size = 3.5pt, mark=+, mark options={solid, magenta}]
  table[row sep=crcr]{%
0.015625	0.000523897352670882\\
0.00390625	0.000130755537175136\\
0.000976562499999999	3.22760518583956e-05\\
0.000244140625	7.85431813516792e-06\\
6.103515625e-05	1.78209711992403e-06\\
};
%\addlegendentry{Phiphi}

%\addplot [color=black, dotted, line width=1pt,forget plot]
%  table[row sep=crcr]{%
%0.015625	0.000456216862700552\\
%6.103515625e-05	1.78209711992403e-06\\
%};

\addplot [color=green, only marks, line width=1pt, mark size = 3.5pt, mark=asterisk, mark options={solid, green}]
  table[row sep=crcr]{%
0.015625	0.000316959793468632\\
0.00390625	7.90372405817723e-05\\
0.000976562499999999	1.93372190048657e-05\\
0.000244140625	4.66836597711959e-06\\
6.103515625e-05	1.03370544296812e-06\\
};
%\addlegendentry{Phiphi (phi2)}

\addplot [color=black, dotted, line width=1pt, forget plot]
  table[row sep=crcr]{%
0.015625	0.000264628593399838\\
6.103515625e-05	1.03370544296812e-06\\
};
\end{axis}

\end{tikzpicture}%
  \caption{Observed decay rates of the local errors (in the
    infinity norm) of the split approximations to
    $\tau\varphi_1(\tau (A+B))v$ and $\tau\varphi_2(\tau(A+B))v$ for
    the different functions $v=v(x,y)$, see \cref{tab:loc}.
    The slope of the dotted line is one, that of the dashed line is two
    and that of the solid line is three.}
  \label{fig:localerrors}
\end{figure}
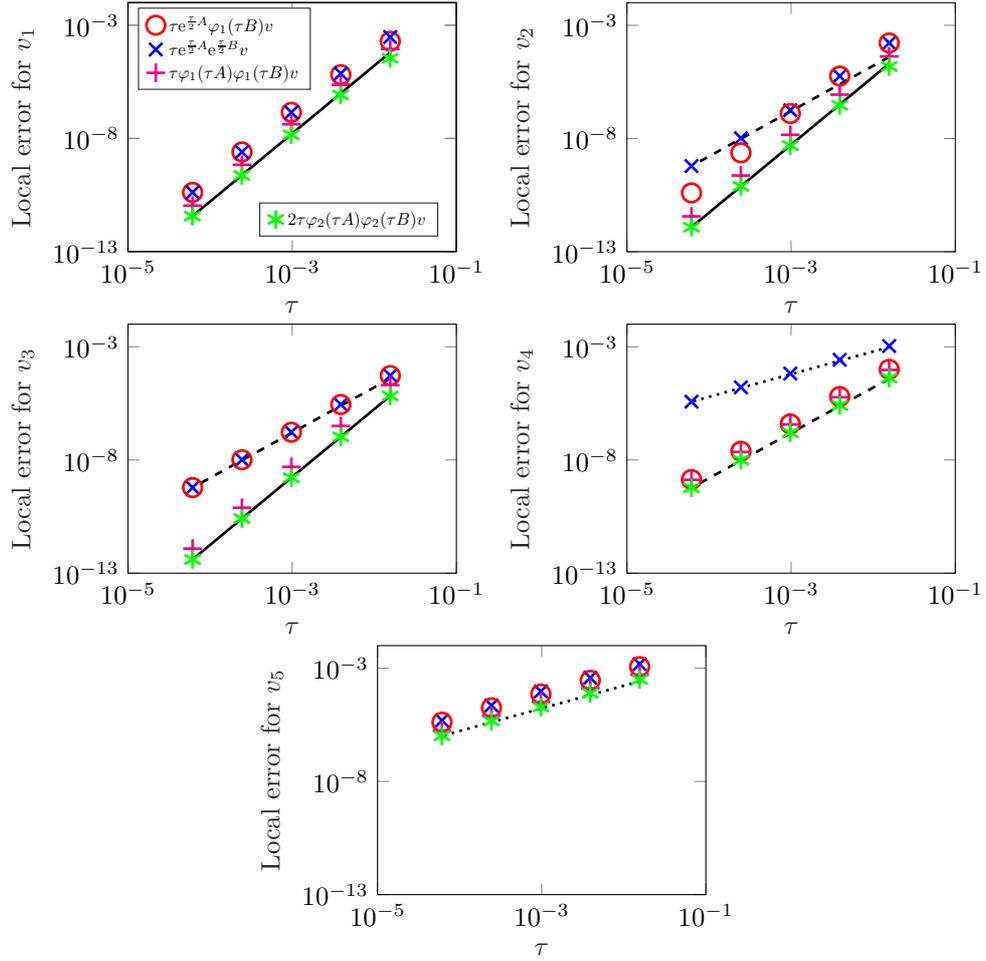
\section{Convergence analysis of two split versions
  of ETD2RK}\label{sec:glob}
In this section, we apply the local error bounds that we have found
in \cref{sec:loc}
to show the convergence (in the stiff sense)
of two split versions of a
popular second-order exponential Runge--Kutta integrator, also known in the
literature as ETD2RK. To this aim, we consider the abstract evolution
equation \cref{eq:AODE}, for which the integrator can be written
as~\cite[Tableau 5.3, $c_2=1$]{HO05bis}
\begin{subequations}
\begin{equation}\label{eq:ETD2RKg}
\begin{aligned}
  \widetilde{U}_{n2}&=\rme^{\tau(A+B)}\widetilde{u}_n+
  \tau\varphi_1(\tau(A+B))g(\widetilde{u}_n)\\
  \widetilde{u}_{n+1}&=\widetilde{U}_{n2}+
  \tau\varphi_2(\tau(A+B))
  (g(\widetilde{U}_{n2})-g(\widetilde{u}_n)),
\end{aligned}
\end{equation}
or equivalently as
\begin{equation}\label{eq:ETD2RKf}
\begin{aligned}
  \widetilde{U}_{n2}&=\widetilde{u}_n + \tau\varphi_1(\tau(A+B))f(\widetilde{u}_n),\\
  \widetilde{u}_{n+1}&=\widetilde{U}_{n2}+
  \tau\varphi_2(\tau(A+B))
  (g(\widetilde{U}_{n2})-g(\widetilde{u}_n)).
\end{aligned}
\end{equation}
\end{subequations}
Here  $\widetilde u_n\approx u(t_n)$, with $t_n=n\tau$ and
$\tau = T/N$ ($N$ is the total number of time steps).
To study the global error we make the following assumptions.
\begin{assumption}\label{assum:analytical}
  Let $X$ be a Banach space with norm $\norm{\cdot}$.
  Let
$A \colon \mathcal{D}(A)\subset X \rightarrow X$ and
$B \colon \mathcal{D}(B)\subset X \rightarrow X$ be linear operators.
We assume
that $A$, $B$ and $A+B$ generate analytic
semigroups $\{\rme^{tA}\}_{t \geq 0}$, $\{\rme^{tB}\}_{t \geq 0}$
and $\{\rme^{t(A+B)}\}_{t \geq 0}$.
\end{assumption}
Note that \cref{assum:analytical} implies that the \emph{parabolic smoothing} bound
\begin{equation*}
  \norm{(-tA)^\alpha\rme^{tA}}\le C,\quad \alpha>0,\quad t\in[0,T]
\end{equation*}
holds (similarly for the operators $B$ and $A+B$).
\begin{assumption}\label{assum:Lipschitz}
We assume that equation~\cref{eq:AODE} possesses a sufficiently
  smooth solution $u\colon[0,T]\to X$ with derivatives in $X$ and
  that $g\colon X\to X$ is twice Fr\'echet differentiable,
with uniformly bounded derivatives.
This in particular implies that $g$ is a Lipschitz function.
\end{assumption}
\subsection{Split exponential integrator ERK2L}
We start by considering the following split version of
integrator~\cref{eq:ETD2RKg}
\begin{equation}\label{eq:ETD2RKdsg}
\begin{aligned}
  U_{n2}&=\rme^{\tau A}\rme^{\tau B}u_n+
  \tau\varphi_1(\tau A)\varphi_1(\tau B)g(u_n),\\
  u_{n+1}&=U_{n2}+
  2\tau\varphi_2(\tau A)\varphi_2(\tau B)
  (g(U_{n2})-g(u_n)),
\end{aligned}
\end{equation}
that we call ERK2L. The same idea was employed in~\cite[Formula 22]{CC24} to obtain
a directional split version of the exponential Euler method.
Notice that integrator \cref{eq:ETD2RKdsg} is exact for \emph{linear}
homogeneous problems (that is, $g\equiv0$).
\begin{theorem}\label{thm:ERK2L}
  Let \cref{assum:commut,assum:analytical,assum:Lipschitz}
  be valid. Then,
  the following bounds on the global error of scheme~\cref{eq:ETD2RKdsg}, i.e.,
  \begin{equation*}
e_{n,\mathrm{L}}=\norm{u_n-u(t_n)},\quad 0\le n\tau \leq T,
  \end{equation*}
hold:
\begin{enumerate}[label=(\alph*)]
\item if $g(u(t))\in\mathcal{D}(BA)$, then
  $e_{n,\mathrm{L}}\leq C\tau^2$;
\item if $g(u(t))\in\mathcal{D}(A)$,
  then $e_{n,\mathrm{L}}\leq C\tau^2(1+\lvert\log \tau\rvert)$;\label{caseb}
  \item if $g(u(t))\not\in\mathcal{D}(A)$, then
  $e_{n,\mathrm{L}}\leq C\tau$.
\end{enumerate}
The constant $C$ is independent of $n$ and $\tau$, but
it depends on $T$.
\end{theorem}
\begin{proof}
Let $h(t)=g(u(t))$. We
consider the variation-of-constants formula~\cref{eq:voc} with $t=t_{n+1}$
and the Taylor expansion
\begin{equation}\label{eq:hexpand}
  h(t_n+s) = h(t_n) + h'(t_n)s + \int_0^s(s-\sigma)h''(t_n+\sigma)d\sigma.
\end{equation}
Then, for the stage $U_{n2}$ (which approximates
$u(t_n+\tau)=u(t_{n+1})$) we have
\begin{equation*}
  U_{n2}-u(t_{n+1}) = \rme^{\tau(A+B)}\epsilon_n + \sum_{i=1}^4 \rho_{n,i},
\end{equation*}
where
\begin{equation*}
\begin{aligned}
  \epsilon_n &= u_n - u(t_n), \\
  \rho_{n,1} &= \tau\big(\varphi_1(\tau A)\varphi_1(\tau B) - \varphi_1(\tau(A+B))\big)h(t_n), \\
  \rho_{n,2} &= \tau\varphi_1(\tau A)\varphi_1(\tau B) \big(g(u_n) - h(t_n)\big), \\
  \rho_{n,3} &= - \tau^2\varphi_2(\tau(A+B))h'(t_n), \\
  \rho_{n,4} &= -\int_0^\tau\rme^{(\tau-s)(A+B)}\int_0^s(s-\sigma) h''(t_n+\sigma)d\sigma ds.
\end{aligned}
\end{equation*}
Note that $\rho_{n,1}$ satisfies
\begin{equation*}
\norm{\rho_{n,1}}\leq C\tau^\alpha \quad \text{with} \quad
\alpha=
\begin{cases}
3 & \text{if } h(t_n)\in\mathcal{D}(BA) \\
2 & \text{if } h(t_n)\in\mathcal{D}(A) \\
1 & \text{if } h(t_n)\notin\mathcal{D}(A) \\
\end{cases}
\end{equation*}
thanks to~\cref{prop:phi1phi1}.
Also, $\norm{\rho_{n,2}}\leq C\tau\norm{\epsilon_n}$ (since $g$ is Lipschitz) and
$\norm{\rho_{n,4}}\leq C\tau^3$. Then, for estimating the global error
\begin{equation*}
  \epsilon_{n+1}=u_{n+1}-u(t_{n+1}) = U_{n2}-u(t_{n+1}) + 2\tau\varphi_2(\tau A)\varphi_2(\tau B)\big(g(U_{n2})-g(u_n)\big)
\end{equation*}
we need bounds on the last term. To this aim, we write
\begin{equation*}
  g(U_{n2})-g(u_n) = \big(g(U_{n2}) - h(t_{n+1})\big) + \big(h(t_n)-g(u_n)\big) + \big(h(t_{n+1})-h(t_n)\big),
\end{equation*}
thus getting
\begin{equation*}
2\tau\varphi_2(\tau A)\varphi_2(\tau B)\big(g(U_{n2})-g(u_n)\big) = \sum_{i=5}^9 \rho_{n,i},
\end{equation*}
where
\begin{equation}\label{eq:rhonj}
  \begin{aligned}
    \rho_{n,5} &= 2\tau\varphi_2(\tau A)\varphi_2(\tau B)\big(g(U_{n2}) - h(t_{n+1})\big) ,\\
    \rho_{n,6} &= 2\tau\varphi_2(\tau A)\varphi_2(\tau B)\big(h(t_n)-g(u_n)\big), \\
    \rho_{n,7} &= \tau\big(2\varphi_2(\tau A)\varphi_2(\tau B) - \varphi_2(\tau(A+B))\big)\big(h(t_{n+1})-h(t_n)\big),\\
    \rho_{n,8} &=\tau^2\varphi_2(\tau(A+B))h'(t_n), \\
    \rho_{n,9} &= \tau\varphi_2(\tau(A+B))\int_0^\tau (\tau-\sigma)h''(t_n+\sigma)d\sigma.
  \end{aligned}
\end{equation}
Here we used again formula~\cref{eq:hexpand}, now with $s=\tau$. Note that, using
the fact that $g$ is Lipschitz, the term $\rho_{n,5}$ satisfies
$\norm{\rho_{n,5}}\leq C\tau\norm{\epsilon_n} + C\tau^{\min\{3,\alpha+1\}}$,
$\norm{\rho_{n,6}}\leq C\tau\norm{\epsilon_n}$, $\norm{\rho_{n,7}}\leq C\tau^{\alpha}$
(see~\cref{prop:phi2phi2}) and $\norm{\rho_{n,9}}\leq C\tau^3$.
Therefore, the recursion of the global error is
\begin{equation*}
\epsilon_{n+1} = \rme^{\tau(A+B)}\epsilon_n + \sum_{\substack{i=1\\i\neq 3,8}}^9 \rho_{n,i}
\end{equation*}
since $\rho_{n,8}=-\rho_{n,3}$. Then, if $h(t)\in\mathcal{D}(BA)$ we have
\begin{equation*}
  \norm{\epsilon_n} \leq C\tau\sum_{k=1}^{n-1} \norm{\epsilon_k} + C\tau^2.
\end{equation*}
Similarly, by exploiting the parabolic smoothing
property for the terms related to $\rho_{n,1}$ and $\rho_{n,7}$ we get
\begin{equation*}
  \norm{\epsilon_n} \leq C\tau\sum_{k=1}^{n-1} \norm{\epsilon_k} +
  C\tau^2(1+\lvert\log \tau \rvert) \quad \text{and} \quad
  \norm{\epsilon_n} \leq C\tau\sum_{k=1}^{n-1} \norm{\epsilon_k} + C\tau
\end{equation*}
for the cases $h(t)\in\mathcal{D}(A)$ and
$h(t)\notin\mathcal{D}(A)$, respectively. Finally, the application of a discrete
Gronwall lemma \cite[Lemma 2.15]{HO10} leads to the result.
\end{proof}

\subsection{Split exponential integrator ERK2}
We now consider the following split version of integrator~\cref{eq:ETD2RKf}
\begin{equation}\label{eq:ETD2RKdsf}
\begin{aligned}
  U_{n2}&=u_n+
  \tau\varphi_1(\tau A)\varphi_1(\tau B)f(u_n),\\
  u_{n+1}&=U_{n2}+
  2\tau\varphi_2(\tau A)\varphi_2(\tau B)
  (g(U_{n2})-g(u_n)),
\end{aligned}
\end{equation}
that we call ERK2. This is the directional split exponential integrator
employed, e.g., in~\cite{CC24bis}.
We start by studying the linear stability of
scheme~\cref{eq:ETD2RKdsf}, i.e., the stability of the recursion
\begin{equation*}
  u_{n+1} = R(\tau A,\tau B) u_n
\end{equation*}
with
\begin{equation*}
   R(\tau A,\tau B) = I + \tau\varphi_1(\tau A)\varphi_1(\tau B)(A+B).
\end{equation*}
To avoid the intricacies of functional analysis, we make the following assumptions.
\begin{assumption}\label{assum:diag}
We assume that there exists a bounded operator $V$, with bounded inverse,
such that the commuting operators $A$ and $B$ satisfy
\begin{equation*}
  A = V^{-1} M_{q_A} V \quad \text{and} \quad B=V^{-1} M_{q_B} V.
\end{equation*}
Here, $M_{q_A}$ and $M_{q_B}$ denote multiplication operators with functions
$q_A$ and $q_B$ having values in $(-\infty,0]$.
\end{assumption}
Note that, since $A$ and $B$ commute, they are transformed by the same operator
$V$.
\begin{assumption}\label{assum:bdddiag}
Let $M_q$ be a multiplication operator. We assume that 
\begin{equation*}
  \norm{M_q}\leq \norm{q}_\infty.
\end{equation*}
\end{assumption}
Note that this is in fact an equality, e.g., in
the space of continuous functions and in $L^p$ spaces, see~\cite{EN00}.
\begin{lemma}\label{lem:stab}
Let~\cref{assum:diag,assum:bdddiag} be valid.
Then, we have
\begin{equation}\label{eq:bdstabf}
  \norm{R(\tau A,\tau B)^n} \leq C
\end{equation}
for $0\leq n\tau \leq T$.
\end{lemma}
\begin{proof}
The result follows easily from the estimate
\begin{equation*}
  \norm{R(\tau A,\tau B)^n} \leq \norm{V^{-1}}\norm{V}
  \norm{R(\tau M_{q_A},\tau M_{q_B})}^n
\end{equation*}
and the fact that $\lvert R(a,b) \rvert \leq 1$
for $a,b \leq 0$.
\end{proof}
We are now in a position to prove the following result.
\begin{theorem}\label{thm:ERK2}
  Let
  \cref{assum:commut,assum:analytical,assum:Lipschitz,assum:diag,assum:bdddiag}
  be valid. Furthermore, assume that
  $g(u(t))-g(u(s)) \in \mathcal{D}(BA)$ for all $0\leq t,s\leq T$.
  Then, the following bound on the global error of
  scheme~\cref{eq:ETD2RKdsf} holds:
  \begin{equation*}
    e_n=\norm{u_n-u(t_n)} \leq C\tau^2, \quad 0\le n\tau \leq T.
  \end{equation*}
The constant $C$ is independent of $n$ and $\tau$,
but it depends on $T$.
\end{theorem}
\begin{proof}
  The proof is very similar to that of~\cref{thm:ERK2L}.
  Setting $h(t)=g(u(t))$, and exploiting the variation-of-constants formula
  with $t=t_{n+1}$, we get for the stage
  \begin{equation*}
    \begin{aligned}
      U_{n2}-u(t_{n+1}) &=R(\tau A,\tau B)(u_n - u(t_n)) \\
      &\quad+\tau(\varphi_1(\tau A)\varphi_1(\tau B)-\varphi_1(\tau(A+B)))f(u(t_n))\\
      &\quad+\tau\varphi_1(\tau A)\varphi_1(\tau B)(g(u_n)-h(t_n))\\
      &\quad-\int_0^\tau \rme^{\tau(A+B)}(h(t_n+s)-h(t_n))ds.
    \end{aligned}
  \end{equation*}
The second term 
is bounded by $C\tau^3$, since $f(u(t))\in \mathcal{D}(BA)$ (it always satisfies
the boundary conditions as $u'(t) = f(u(t))$ does).
For the remaining terms and the rest of the proof, employing the Taylor
expansion~\cref{eq:hexpand} and the stability bound \cref{eq:bdstabf}
in \cref{lem:stab}, we proceed as in the proof of \cref{thm:ERK2L}.
Note that, in this case, the term corresponding to $\rho_{n,5}$
(see formula~\cref{eq:rhonj}) satisfies
$\norm{\epsilon_n}\leq C\tau\norm{\epsilon_n} + C\tau^2$, while $\rho_{n,7}$
is always bounded by $C\tau^3$ (as by assumption
$h(t)-h(s) \in \mathcal{D}(BA)$).
\end{proof}
\begin{remark}\label{rem:diffint}
Although integrators~\cref{eq:ETD2RKg} and \cref{eq:ETD2RKf} are
equivalent, this is \emph{not} the case for \cref{eq:ETD2RKdsg} and \cref{eq:ETD2RKdsf}.
In fact, the split scheme~\cref{eq:ETD2RKdsg} may have order reduction to
one (e.g., if $g(u(t))\notin\mathcal{D}(A)$),
its error constant scales with $g(u(t))$ (see the proof of~\cref{thm:ERK2L} and the experiments
in the next section), and it is exact when $g\equiv 0$.
These assertions are not true for the split scheme~\cref{eq:ETD2RKdsf} which, on the other
hand, in our cases of interest is always second-order accurate (see the numerical experiments in the next section).
\end{remark}

\subsection{Numerical experiments}
In this section, we present some numerical experiments that illustrate the
behavior of the two split integrators \cref{eq:ETD2RKdsg,eq:ETD2RKdsf}
under various circumstances.
We perform the experiments using Matlab R2022a on an Intel Core i7-10750H CPU
(equipped with 16GB of RAM).

We start by considering the following two-dimensional evolution
diffusion-reaction equation in the spatial domain $\Omega=(0,1)^2$
\begin{subequations}\label{eq:semilinear_2d}
  \begin{equation}
  \left\{
  \begin{aligned}
    \partial_t u(t,x,y)&=\Delta u(t,x,y)+
     g(u(t,x,y),x,y),\\
    u(0,x,y)&=16x(1-x)y(1-y),
  \end{aligned}\right.
\end{equation}
subject to homogeneous Dirichlet boundary conditions.
The nonlinear term is
\begin{equation}\label{eq:g2d}
  g(u(t,x,y),x,y)=\frac{\kappa q(x,y)^p}{1+u(t,x,y)^2},
  \quad p\in\{0,1\},
\end{equation}
\end{subequations}
with $\kappa$ a real parameter and
$q(x,y)=16x(1-x)y(1-y)$. Clearly, the nonlinear term \emph{does not}
satisfy the boundary conditions if $p=0$, while it does if $p=1$.
The Laplacian is approximated by standard second-order
finite differences
with $250^2$ interior discretization points. The resulting
matrix is $(I\otimes D+D\otimes I)$, where $I$ is the identity
matrix, $\otimes$ denotes the Kronecker product
and $D$ is the
approximation matrix of the one-dimensional second derivative
operator with homogeneous Dirichlet boundary conditions.
The actions of the matrix exponentials and the
matrix $\varphi$ functions
on a vector $v$ required by the integrators ERK2L and ERK2 are efficiently
computed by directionally splitting the discretized Laplacian
and exploiting
the equivalence
\begin{equation*}
  \varphi_\ell(\tau(I\otimes D))\varphi_{\ell}(\tau(D\otimes I))v=
  \left(\varphi_\ell(\tau D)\otimes\varphi_\ell(\tau D)\right)v=
  \mathrm{vec}\left(\varphi_\ell(\tau D)V\varphi_\ell(\tau D)^{\sf T}\right).
\end{equation*}
Here $\mathrm{vec}$ denotes the operator which stacks the
input matrix by columns and $V$, on the other hand, is a matrix
satisfying $\mathrm{vec}(V)=v$.
Note that the realization of the last expression
requires only to compute \emph{small} matrix functions (i.e., $\varphi$ functions
of the discretization of the \emph{one-dimensional} second derivative operator) and
two dense matrix-matrix products.
The matrix functions are computed once and for all before the time integration
using a standard Pad\'e approximation with modified scaling and
squaring~\cite{SW09}.
We compare the integrators with the (unsplit) ETD2RK
method \cref{eq:ETD2RKf}, in which the required actions of matrix
functions $\varphi_\ell(\tau(I\otimes D+D\otimes I))v$ are computed
at each time step
using the adaptive Krylov subspace solver presented in~\cite{GRT18}.

The results for $\kappa=2$ and $\kappa=0.02$, with final time $T=0.1$
and number of time steps $N=2^j$, with $j=2,4,6,8,10$,
are collected in \cref{fig:global_2d}.
As stated in \cref{thm:ERK2L,thm:ERK2}, when the nonlinear term \cref{eq:g2d}
satisfies the boundary conditions ($p=1$, top row in the figure)
we observe second-order convergence for both the split schemes. Otherwise
($p=0$, bottom row), the scheme ERK2L exhibits order reduction to one.
Moreover, in contrast to ERK2,
we remark that the errors of the schemes ETD2RK and
ERK2L scale according to the parameter $\kappa$ in the nonlinear term.
This is expected, since
the remainders in the global errors of both schemes are proportional to
the nonlinear term (see \cref{rem:diffint}).
Finally, the average wall-clock time per time step of the
split schemes is about 4 milliseconds, while
that of the unsplit scheme is roughly 30 milliseconds.
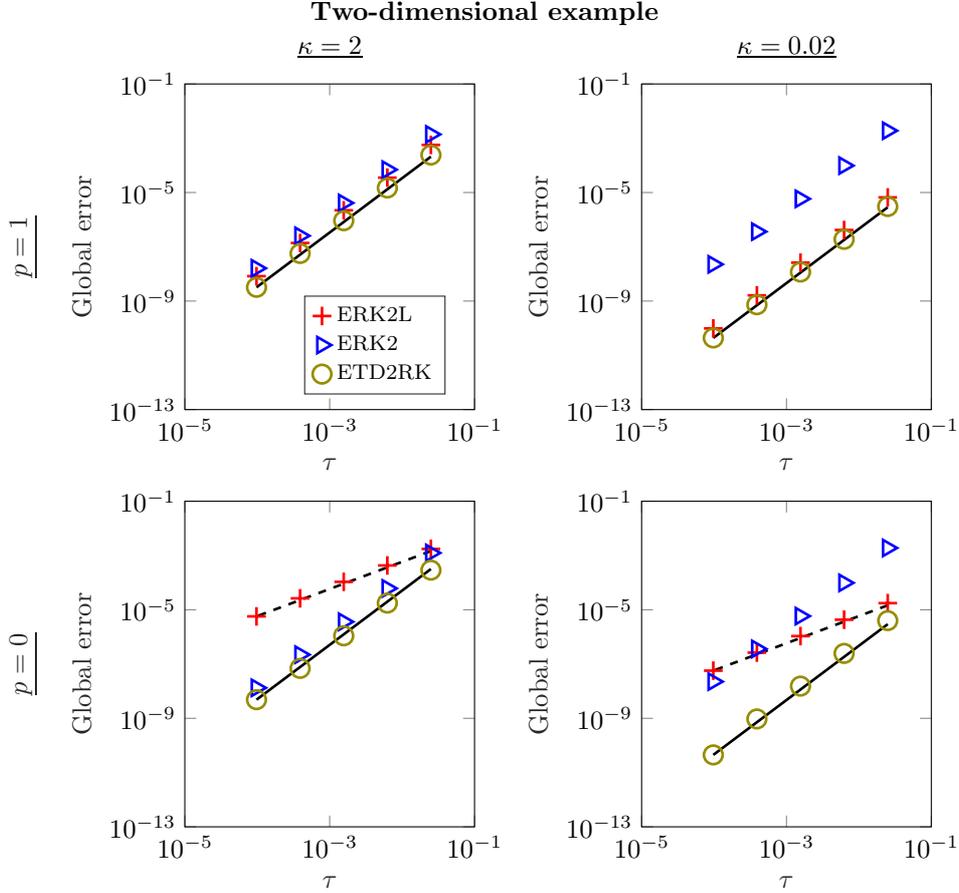
\begin{figure}[!ht]
  \centering
  \textbf{Two-dimensional example}

\begin{tikzpicture}

\begin{axis}[%
width=1.5in,
height=1.7in,
at={(0.758in,0.481in)},
scale only axis,
xmode=log,
xmin=1e-5,
xmax=1e-1,
xminorticks=true,
xlabel style={font=\color{white!15!black}},
xlabel={$\tau$},
xtick = {1e-5,1e-3,1e-1},
xticklabels = {$10^{-5}$,$10^{-3}$,$10^{-1}$},
ymode=log,
ymin=1e-13,
ymax=1e-1,
yminorticks=true,
ylabel style={font=\color{white!15!black}, align=center},
ylabel={$\underline{p=1}$\\\\Global error},
title={$\underline{\kappa=2}$},
axis background/.style={fill=white},
legend style={legend cell align=left, align=left, draw=white!15!black, at={(0.9,0.35)},font=\footnotesize}
]
\addplot [color=red, line width=1.0pt, only marks, mark size=3.5pt, mark=+, mark options={solid, red}]
  table[row sep=crcr]{%
0.025	0.000567670013265437\\
0.00625	3.53413346513842e-05\\
0.0015625	2.20855619689763e-06\\
0.000390625	1.37708323744156e-07\\
9.76562499999999e-05	8.26841584089521e-09\\
};
\addlegendentry{ERK2L}

\addplot [color=blue, line width=1.0pt, only marks, mark size=3.5pt, mark=triangle, mark options={solid, rotate=270, blue}]
  table[row sep=crcr]{%
0.025	0.00138672920411181\\
0.00625	6.96215643753994e-05\\
0.0015625	4.11759065205297e-06\\
0.000390625	2.54156409168482e-07\\
9.76562500000001e-05	1.61673316678712e-08\\
};
\addlegendentry{ERK2}

\addplot [color=olive, line width=1.0pt, only marks, mark size=3.5pt, mark=o, mark options={solid, olive}]
  table[row sep=crcr]{%
0.025	0.00024106152240555\\
0.00625	1.46989234908468e-05\\
0.0015625	9.13694100435068e-07\\
0.000390625	5.67039684373772e-08\\
9.76562499999999e-05	3.2056574894579e-09\\
};
\addlegendentry{ETD2RK}

%\addplot [color=black, line width=1pt, forget plot]
%  table[row sep=crcr]{%
%0.025	0.00054187890054891\\
%9.76562499999999e-05	8.26841584089521e-09\\
%};
%\addplot [color=black, line width=1pt, forget plot]
%  table[row sep=crcr]{%
%0.025	0.00105954224818561\\
%9.76562500000001e-05	1.61673316678712e-08\\
%};
\addplot [color=black, line width=1pt, forget plot]
  table[row sep=crcr]{%
0.025	0.000210085969229112\\
9.76562499999999e-05	3.2056574894579e-09\\
};
\end{axis}

\end{tikzpicture}%
\begin{tikzpicture}

\begin{axis}[%
width=1.5in,
height=1.7in,
at={(0.758in,0.481in)},
scale only axis,
xmode=log,
xmin=1e-5,
xmax=1e-1,
xminorticks=true,
xlabel style={font=\color{white!15!black}},
xlabel={$\tau$},
xtick = {1e-5,1e-3,1e-1},
xticklabels = {$10^{-5}$,$10^{-3}$,$10^{-1}$},
ymode=log,
ymin=1e-13,
ymax=1e-1,
yminorticks=true,
ylabel style={font=\color{white!15!black}},
ylabel={Global error},
title={$\underline{\kappa=0.02}$},
axis background/.style={fill=white},
legend style={legend cell align=left, align=left, draw=white!15!black, at={(0.9,0.35)},font=\footnotesize}
]
\addplot [color=red, line width=1.0pt, only marks, mark size=3.5pt, mark=+, mark options={solid, red}]
  table[row sep=crcr]{%
0.025	6.60577714817201e-06\\
0.00625	4.14765352646507e-07\\
0.0015625	2.59621928400121e-08\\
0.000390625	1.62005583947966e-09\\
9.76562499999999e-05	9.79155923008788e-11\\
};
%\addlegendentry{ETD2RKds $g$ form}

\addplot [color=blue, line width=1.0pt, only marks, mark size=3.5pt, mark=triangle, mark options={solid, rotate=270, blue}]
  table[row sep=crcr]{%
0.025	0.00189346597418102\\
0.00625	9.80600461235049e-05\\
0.0015625	5.84986786683616e-06\\
0.000390625	3.61416606059484e-07\\
9.76562500000001e-05	2.25256300556698e-08\\
};
%\addlegendentry{ETD2RKds $f$ form}

\addplot [color=olive, line width=1.0pt, only marks, mark size=3.5pt, mark=o, mark options={solid, olive}]
  table[row sep=crcr]{%
0.025	3.04991708707791e-06\\
0.00625	1.88197134404211e-07\\
0.0015625	1.17488638817775e-08\\
0.000390625	7.3199374339694e-10\\
9.76562499999999e-05	4.36804481473984e-11\\
};
%\addlegendentry{ETD2RK}

%\addplot [color=black, line width=1pt, forget plot]
%  table[row sep=crcr]{%
%0.025	6.41699625703041e-06\\
%9.76562499999999e-05	9.79155923008788e-11\\
%};
%\addplot [color=black, line width=1pt, forget plot]
%  table[row sep=crcr]{%
%0.025	0.00147623969132837\\
%9.76562500000001e-05	2.25256300556698e-08\\
%};
\addplot [color=black, line width=1pt, forget plot]
  table[row sep=crcr]{%
0.025	2.86264184978791e-06\\
9.76562499999999e-05	4.36804481473984e-11\\
};
\end{axis}

\end{tikzpicture}%
\\
\begin{tikzpicture}

\begin{axis}[%
width=1.5in,
height=1.7in,
at={(0.758in,0.481in)},
scale only axis,
xmode=log,
xmin=1e-5,
xmax=1e-1,
xminorticks=true,
xlabel style={font=\color{white!15!black}},
xlabel={$\tau$},
xtick = {1e-5,1e-3,1e-1},
xticklabels = {$10^{-5}$,$10^{-3}$,$10^{-1}$},
ymode=log,
ymin=1e-13,
ymax=1e-1,
yminorticks=true,
ylabel style={font=\color{white!15!black}, align=center},
ylabel={$\underline{p=0}$\\\\Global error},
axis background/.style={fill=white},
legend style={legend cell align=left, align=left, draw=white!15!black, at={(0.9,0.35)},font=\footnotesize}
]
\addplot [color=red, line width=1.0pt, only marks, mark size=3.5pt, mark=+, mark options={solid, red}]
  table[row sep=crcr]{%
0.025	0.00174174166760439\\
0.00625	0.000429593529073487\\
0.0015625	0.000106819069769652\\
0.000390625	2.63121809168487e-05\\
9.76562500000001e-05	5.69568396781647e-06\\
};
%\addlegendentry{ETD2RKds $g$ form}

\addplot [color=blue, line width=1.0pt, only marks, mark size=3.5pt, mark=triangle, mark options={solid, rotate=270, blue}]
  table[row sep=crcr]{%
0.025	0.00124280078187533\\
0.00625	6.12489803847871e-05\\
0.0015625	3.60592641529544e-06\\
0.000390625	2.21493879348777e-07\\
9.76562500000001e-05	1.32485337700494e-08\\
};
%\addlegendentry{ETD2RKds $f$ form}

\addplot [color=olive, line width=1.0pt, only marks, mark size=3.5pt, mark=o, mark options={solid, olive}]
  table[row sep=crcr]{%
0.025	0.000288324895363789\\
0.00625	1.75970387186752e-05\\
0.0015625	1.09497660982783e-06\\
0.000390625	6.89137294451747e-08\\
9.76562499999999e-05	4.84813666989937e-09\\
};
%\addlegendentry{ETD2RK}

\addplot [color=black, dashed, line width=1pt, forget plot]
  table[row sep=crcr]{%
0.025	0.00145809509576102\\
9.76562500000001e-05	5.69568396781647e-06\\
};
%\addplot [color=black, line width=1pt, forget plot]
%  table[row sep=crcr]{%
%0.025	0.000868255909153958\\
%9.76562500000001e-05	1.32485337700494e-08\\
%};
\addplot [color=black, line width=1pt, forget plot]
  table[row sep=crcr]{%
0.025	0.000317727484798524\\
9.76562499999999e-05	4.84813666989937e-09\\
};
\end{axis}

\end{tikzpicture}%
\begin{tikzpicture}

\begin{axis}[%
width=1.5in,
height=1.7in,
at={(0.758in,0.481in)},
scale only axis,
xmode=log,
xmin=1e-5,
xmax=1e-1,
xminorticks=true,
xlabel style={font=\color{white!15!black}},
xlabel={$\tau$},
xtick = {1e-5,1e-3,1e-1},
xticklabels = {$10^{-5}$,$10^{-3}$,$10^{-1}$},
ymode=log,
ymin=1e-13,
ymax=1e-1,
yminorticks=true,
ylabel style={font=\color{white!15!black}},
ylabel = {Global error},
axis background/.style={fill=white},
legend style={legend cell align=left, align=left, draw=white!15!black, at={(0.9,0.35)},font=\footnotesize}
]
\addplot [color=red, line width=1.0pt, only marks, mark size=3.5pt, mark=+, mark options={solid, red}]
  table[row sep=crcr]{%
0.025	1.75224163311872e-05\\
0.00625	4.29879573120599e-06\\
0.0015625	1.06829032924134e-06\\
0.000390625	2.63124563063281e-07\\
9.76562500000001e-05	5.69569418945955e-08\\
};
%\addlegendentry{ETD2RKds $g$ form}

\addplot [color=blue, line width=1.0pt, only marks, mark size=3.5pt, mark=triangle, mark options={solid, rotate=270, blue}]
  table[row sep=crcr]{%
0.025	0.00189152650643562\\
0.00625	9.79449320473723e-05\\
0.0015625	5.84281875318226e-06\\
0.000390625	3.60990021763819e-07\\
9.76562500000001e-05	2.25108823248554e-08\\
};
%\addlegendentry{ETD2RKds $f$ form}

\addplot [color=olive, line width=1.0pt, only marks, mark size=3.5pt, mark=o, mark options={solid, olive}]
  table[row sep=crcr]{%
0.025	3.98147426544647e-06\\
0.00625	2.45895503181437e-07\\
0.0015625	1.53237112721083e-08\\
0.000390625	9.46512646038398e-10\\
9.76562499999999e-05	4.4835052337433e-11\\
};
%\addlegendentry{ETD2RK}

\addplot [color=black, dashed, line width=1pt, forget plot]
  table[row sep=crcr]{%
0.025	1.45809771250165e-05\\
9.76562500000001e-05	5.69569418945955e-08\\
};
%\addplot [color=black, line width=1pt, forget plot]
%  table[row sep=crcr]{%
%0.025	0.00147527318404172\\
%9.76562500000001e-05	2.25108823248554e-08\\
%};
\addplot [color=black, line width=1pt, forget plot]
  table[row sep=crcr]{%
0.025	2.938309989986e-06\\
9.76562499999999e-05	4.4835052337433e-11\\
};
\end{axis}

\end{tikzpicture}%
  \caption{Observed decay rates of the global errors (at $T=0.1$ in the
    infinity norm)
    of ERK2L~\cref{eq:ETD2RKdsg}, ERK2~\cref{eq:ETD2RKdsf} and
    ETD2RK~\cref{eq:ETD2RKf} for
    the two-dimensional example~\cref{eq:semilinear_2d}.
    The slope of the dashed line is one and that of the solid line is two.}
   \label{fig:global_2d}
\end{figure}

We finally repeat the example in a three-dimensional scenario, i.e., we solve
\begin{subequations}\label{eq:semilinear_3d}
\begin{equation}
  \left\{
  \begin{aligned}
    \partial_t u(t,x,y,z)&=\Delta u(t,x,y,z)+
    g(u(t,x,y,z),x,y,z),\\
     u(0,x,y,z)&=64x(1-x)y(1-y)z(1-z),
  \end{aligned}\right.
\end{equation}
in the spatial domain $\Omega=(0,1)^3$,
subject to homogeneous Dirichlet boundary conditions.
The nonlinear term is
\begin{equation}
  g(u(t,x,y,z),x,y,z)=\frac{\kappa q(x,y,z)^p}{1+u(t,x,y,z)^2},
  \quad p\in\{0,1\},
\end{equation}
\end{subequations}
with $q(x,y,z)=64x(1-x)y(1-y)z(1-z)$.
The Laplacian is approximated by finite differences, this time with
$250^3$ interior discretization points, resulting in the matrix
$(I\otimes I\otimes D+I\otimes D\otimes I+D\otimes I\otimes I)$.
In this case, after directional splitting,
the action of the matrix functions on a vector $v$
is realized by
\begin{multline*}
  \varphi_\ell(\tau(I\otimes I\otimes D))
  \varphi_\ell(\tau(I\otimes D\otimes I))
  \varphi_{\ell}(\tau(D\otimes I\otimes I))v=\\
  \mathrm{vec}\left(V\times_1\varphi_\ell(\tau D)\times_2\varphi_\ell(\tau D)
  \times_3\varphi_\ell(\tau D)\right),
\end{multline*}
where $V$ is an order-3 tensor
satisfying $\mathrm{vec}(V)=v$
(here $\mathrm{vec}$ is the operator which stacks the input tensor by columns)
and $\times_\mu$ denotes the tensor-matrix product along direction $\mu$.
Note that the right-hand side is efficiently realized by three dense
matrix-matrix products (level 3 BLAS operations).
A reader not familiar with the above formalism is invited to
consult~\cite{CC24bis,CCEOZ22,CCZ23bis,C24}.

From the results collected in \cref{fig:global_3d} we draw
similar conclusions to that of the two-dimensional case, i.e., the split
integrators show the expected behavior and convergence order. In this
three-dimensional scenario, the average wall-clock time per time step for the
split methods is about $1.4$ seconds, while
the ETD2RK method did not finish the computation
of the first time step
because of too large memory requirements.
\begin{figure}[!ht]
  \centering
  \textbf{Three-dimensional example}

\begin{tikzpicture}

\begin{axis}[%
width=1.5in,
height=1.7in,
at={(0.758in,0.481in)},
scale only axis,
xmode=log,
xmin=1e-5,
xmax=1e-1,
xminorticks=true,
xlabel style={font=\color{white!15!black}},
xlabel={$\tau$},
xtick = {1e-5,1e-3,1e-1},
xticklabels = {$10^{-5}$,$10^{-3}$,$10^{-1}$},
ymode=log,
ymin=1e-13,
ymax=1e-1,
yminorticks=true,
ylabel style={font=\color{white!15!black}, align=center},
ylabel={$\underline{p=1}$\\\\Global error},
title={$\underline{\kappa=2}$},
axis background/.style={fill=white},
legend style={legend cell align=left, align=left, draw=white!15!black, at={(0.9,0.35)},font=\footnotesize}
]
\addplot [color=red, line width=1.0pt, only marks, mark size=3.5pt, mark=+, mark options={solid, red}]
  table[row sep=crcr]{%
0.025	0.000997869689627173\\
0.00625	6.25285372145846e-05\\
0.0015625	3.91029166646462e-06\\
0.000390625	2.43552417913961e-07\\
9.76562500000001e-05	1.43274658109815e-08\\
};
\addlegendentry{ERK2L}

\addplot [color=blue, line width=1.0pt, only marks, mark size=3.5pt, mark=triangle, mark options={solid, rotate=270, blue}]
  table[row sep=crcr]{%
0.025	0.00323020357616213\\
0.00625	0.000148088209958361\\
0.0015625	8.58402120446077e-06\\
0.000390625	5.27510057690694e-07\\
9.76562500000001e-05	3.37117719434454e-08\\
};
\addlegendentry{ERK2}

\addplot [color=black, line width=1pt, forget plot]
  table[row sep=crcr]{%
0.025	0.000938964799388487\\
9.76562500000001e-05	1.43274658109815e-08\\
};
%\addplot [color=black, line width=1pt, forget plot]
%  table[row sep=crcr]{%
%0.025	0.00220933468608564\\
%9.76562500000001e-05	3.37117719434454e-08\\
%};
\end{axis}
\end{tikzpicture}%
\begin{tikzpicture}

\begin{axis}[%
width=1.5in,
height=1.7in,
at={(0.758in,0.481in)},
scale only axis,
xmode=log,
xmin=1e-5,
xmax=1e-1,
xminorticks=true,
xlabel style={font=\color{white!15!black}},
xlabel={$\tau$},
xtick = {1e-5,1e-3,1e-1},
xticklabels = {$10^{-5}$,$10^{-3}$,$10^{-1}$},
ymode=log,
ymin=1e-13,
ymax=1e-1,
yminorticks=true,
ylabel style={font=\color{white!15!black}},
ylabel={Global error},
title={$\underline{\kappa=0.02}$},
axis background/.style={fill=white},
legend style={legend cell align=left, align=left, draw=white!15!black, at={(0.9,0.35)},font=\footnotesize}
]
\addplot [color=red, line width=1.0pt, only marks, mark size=3.5pt, mark=+, mark options={solid, red}]
  table[row sep=crcr]{%
0.025	1.04588595807309e-05\\
0.00625	6.5785283992481e-07\\
0.0015625	4.11589652984933e-08\\
0.000390625	2.56382704649782e-09\\
9.76562499999999e-05	1.50854134739475e-10\\
};
%\addlegendentry{ETD2RKds exp u+phi1 g}

\addplot [color=blue, line width=1.0pt, only marks, mark size=3.5pt, mark=triangle, mark options={solid, rotate=270, blue}]
  table[row sep=crcr]{%
0.025	0.00378933609834366\\
0.00625	0.000175668157776893\\
0.0015625	1.02264159518722e-05\\
0.000390625	6.2811808699409e-07\\
9.76562500000001e-05	3.90962420840357e-08\\
};
%\addlegendentry{ETD2RKds u+phi1 f}

\addplot [color=black, line width=1pt, forget plot]
  table[row sep=crcr]{%
0.025	9.88637657428623e-06\\
9.76562499999999e-05	1.50854134739475e-10\\
};
%\addplot [color=black, line width=1pt, forget plot]
%  table[row sep=crcr]{%
%0.025	0.00256221132121936\\
%9.76562500000001e-05	3.90962420840357e-08\\
%};
\end{axis}
\end{tikzpicture}%
\\
\begin{tikzpicture}

\begin{axis}[%
width=1.5in,
height=1.7in,
at={(0.758in,0.481in)},
scale only axis,
xmode=log,
xmin=1e-5,
xmax=1e-1,
xminorticks=true,
xlabel style={font=\color{white!15!black}},
xlabel={$\tau$},
xtick = {1e-5,1e-3,1e-1},
xticklabels = {$10^{-5}$,$10^{-3}$,$10^{-1}$},
ymode=log,
ymin=1e-13,
ymax=1e-1,
yminorticks=true,
ylabel style={font=\color{white!15!black}, align=center},
ylabel={$\underline{p=0}$\\\\Global error},
axis background/.style={fill=white},
legend style={legend cell align=left, align=left, draw=white!15!black, at={(0.9,0.35)},font=\footnotesize}
]
\addplot [color=red, line width=1.0pt, only marks, mark size=3.5pt, mark=+, mark options={solid, red}]
  table[row sep=crcr]{%
0.025	0.00312348863139994\\
0.00625	0.00077580314896539\\
0.0015625	0.000193157375819211\\
0.000390625	4.71420507622423e-05\\
9.76562500000001e-05	1.11722379227937e-05\\
};
%\addlegendentry{ETD2RKds exp u+phi1 g}

\addplot [color=blue, line width=1.0pt, only marks, mark size=3.5pt, mark=triangle, mark options={solid, rotate=270, blue}]
  table[row sep=crcr]{%
0.025	0.00293653068737174\\
0.00625	0.000132219504785841\\
0.0015625	1.08553640292493e-05\\
0.000390625	7.04562790057546e-07\\
9.76562500000001e-05	2.73096213698753e-08\\
};
%\addlegendentry{ETD2RKds u+phi1 f}

\addplot [color=black, dashed, line width=1pt, forget plot]
  table[row sep=crcr]{%
0.025	0.0028600929082352\\
9.76562500000001e-05	1.11722379227937e-05\\
};
\addplot [color=black, line width=1pt, forget plot]
  table[row sep=crcr]{%
0.025	0.00178976334609615\\
9.76562500000001e-05	2.73096213698753e-08\\
};
\end{axis}
\end{tikzpicture}%
\begin{tikzpicture}

\begin{axis}[%
width=1.5in,
height=1.7in,
at={(0.758in,0.481in)},
scale only axis,
xmode=log,
xmin=1e-5,
xmax=1e-1,
xminorticks=true,
xlabel style={font=\color{white!15!black}},
xlabel={$\tau$},
xtick = {1e-5,1e-3,1e-1},
xticklabels = {$10^{-5}$,$10^{-3}$,$10^{-1}$},
ymode=log,
ymin=1e-13,
ymax=1e-1,
yminorticks=true,
ylabel style={font=\color{white!15!black}},
ylabel={Global error},
axis background/.style={fill=white},
legend style={legend cell align=left, align=left, draw=white!15!black, at={(0.9,0.35)},font=\footnotesize}
]
\addplot [color=red, line width=1.0pt, only marks, mark size=3.5pt, mark=+, mark options={solid, red}]
  table[row sep=crcr]{%
0.025	3.12765019659513e-05\\
0.00625	7.75874318706556e-06\\
0.0015625	1.93158783369245e-06\\
0.000390625	4.71421106738183e-07\\
9.76562500000001e-05	1.11722422124655e-07\\
};
%\addlegendentry{ETD2RKds $g$ form}

\addplot [color=blue, line width=1.0pt, only marks, mark size=3.5pt, mark=triangle, mark options={solid, rotate=270, blue}]
  table[row sep=crcr]{%
0.025	0.00378621656675193\\
0.00625	0.000175495726940513\\
0.0015625	1.02138189943882e-05\\
0.000390625	6.25043615996856e-07\\
9.76562500000001e-05	3.66100035009697e-08\\
};
%\addlegendentry{ETD2RKds $f$ form}

\addplot [color=black, dashed, line width=1pt, forget plot]
  table[row sep=crcr]{%
0.025	2.86009400639117e-05\\
9.76562500000001e-05	1.11722422124655e-07\\
};
\addplot [color=black, line width=1pt, forget plot]
  table[row sep=crcr]{%
0.025	0.00239927318943955\\
9.76562500000001e-05	3.66100035009697e-08\\
};
\end{axis}

\end{tikzpicture}%
  \caption{Observed decay rate of the global error (at $T=0.1$ in the
    infinity norm)
    of ERK2L~\cref{eq:ETD2RKdsg} and ERK2~\cref{eq:ETD2RKdsf} for
    the three-dimensional example~\cref{eq:semilinear_3d}.
    The slope of the dashed line is one and that of the solid line is two.}
  \label{fig:global_3d}
\end{figure}
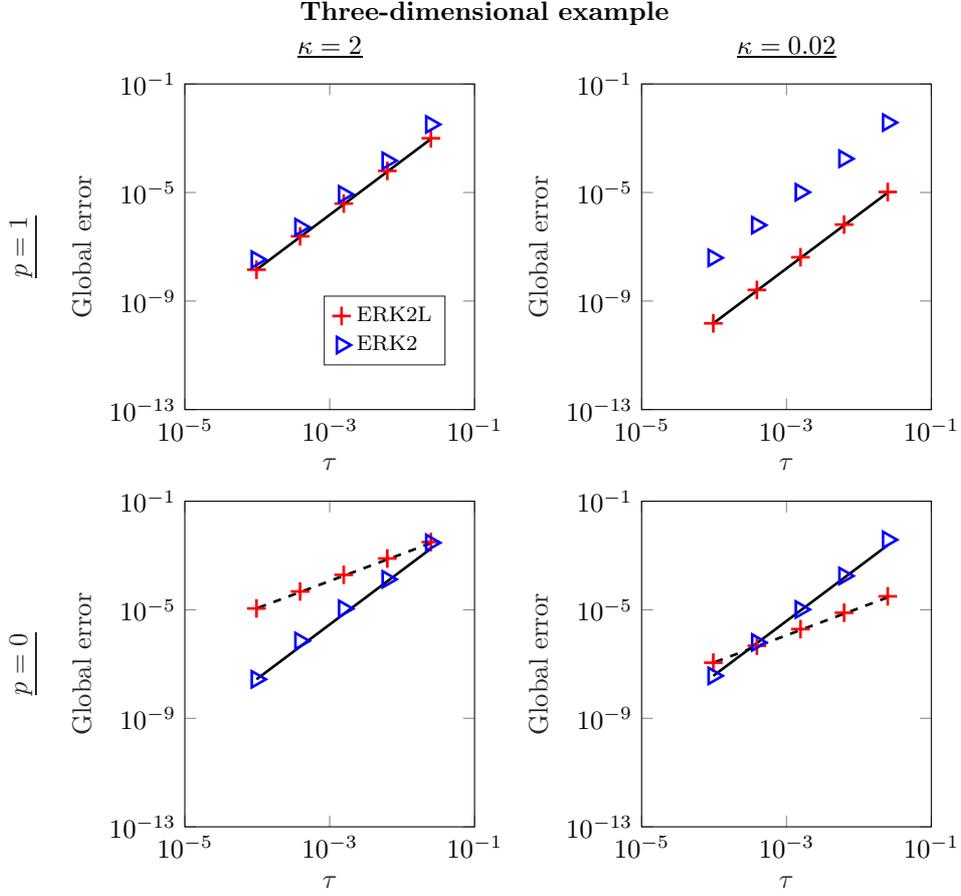

\section{Conclusions and future work}\label{sec:concl}
In this paper, we analyzed split
approximation for the $\varphi$ functions
in an abstract setting.
In particular, we showed that an order reduction
may occur if, for instance, the boundary conditions are not fulfilled by the
involved functions.
We then employed the derived bounds
to prove convergence results for two split versions of
a popular second-order exponential integrator for semilinear
parabolic equations. The accompanying numerical simulations in
two and three spatial dimensions confirm
the theoretical results.

An interesting further development is the study
of high-order  split approximations for
$\varphi$ functions, useful for the effective employment
of exponential schemes of order higher than two.

\section*{Acknowledgments}
Fabio Cassini holds a post-doc fellowship funded by the
Istituto Nazionale di Alta Matematica (INdAM).
Fabio Cassini and Marco Caliari are members of the Gruppo Nazionale
Calcolo Scientifico-Istituto Nazionale di Alta Matematica (GNCS-INdAM).

\bibliographystyle{plain}
\bibliography{CCEO24}

\begin{thebibliography}{10}

\bibitem{AMCR17}
I.~Alonso-Mallo, B.~Cano, and N.~Reguera.
\newblock Avoiding order reduction when integrating linear initial boundary
  value problems with {L}awson methods.
\newblock {\em IMA J. Numer. Anal.}, 37(4):2091--2119, 2017.

\bibitem{CC24}
M.~Caliari and F.~Cassini.
\newblock Direction splitting of $\varphi$-functions in exponential integrators
  for $d$-dimensional problems in {K}ronecker form.
\newblock {\em J. Approx. Softw.}, 1(1):1, 2024.

\bibitem{CC24bis}
M.~Caliari and F.~Cassini.
\newblock A second order directional split exponential integrator for systems
  of advection--diffusion--reaction equations.
\newblock {\em J. Comput. Phys.}, 498, 2024.

\bibitem{CCEOZ22}
M.~Caliari, F.~Cassini, L.~Einkemmer, A.~Ostermann, and F.~Zivcovich.
\newblock A $\mu$-mode integrator for solving evolution equations in
  {K}ronecker form.
\newblock {\em J. Comput. Phys.}, 455, 2022.

\bibitem{CCZ23bis}
M.~Caliari, F.~Cassini, and F.~Zivcovich.
\newblock A $\mu$-mode {BLAS} approach for multidimensional tensor-structured
  problems.
\newblock {\em Numer. Algorithms}, 92(4):2483--2508, 2023.

\bibitem{CR22}
B.~Cano and N.~Reguera.
\newblock How to avoid order reduction when {L}awson methods integrate
  nonlinear initial boundary value problems.
\newblock {\em BIT Numer. Math.}, 62:431--463, 2022.

\bibitem{C24}
F.~Cassini.
\newblock Efficient third order tensor-oriented directional splitting for
  exponential integrators.
\newblock {\em Electron. Trans. Numer. Anal.}, 60:520--540, 2024.

\bibitem{CEM20}
N.~Crouseilles, L.~Einkemmer, and J.~Massot.
\newblock Exponential methods for solving hyperbolic problems with application
  to collisionless kinetic equations.
\newblock {\em J. Comput. Phys.}, 420, 2020.

\bibitem{EO15b}
L.~Einkemmer and A.~Ostermann.
\newblock Overcoming order reduction in diffusion-reaction splitting. {P}art 1:
  {D}irichlet boundary conditions.
\newblock {\em SIAM J. Sci. Comput.}, 37(3):A1577--A1592, 2015.

\bibitem{EO16}
L.~Einkemmer and A.~Ostermann.
\newblock Overcoming order reduction in diffusion-reaction splitting. {P}art 2:
  {O}blique boundary conditions.
\newblock {\em SIAM J. Sci. Comput.}, 38(6):A3741--A3757, 2016.

\bibitem{ETL17}
L.~Einkemmer, M.~Tokman, and J.~Loffeld.
\newblock On the performance of exponential integrators for problems in
  magnetohydrodynamics.
\newblock {\em J. Comput. Phys.}, 330:550--565, 2017.

\bibitem{EN00}
K.-J. Engel and R.~Nagel.
\newblock {\em One-Parameter Semigroups for Linear Evolution Equations}, volume
  194 of {\em Graduate Texts in Mathematics}.
\newblock Springer, 2000.

\bibitem{GRT18}
S.~Gaudreault, G.~Rainwater, and M.~Tokman.
\newblock {KIOPS: A} fast adaptive {K}rylov subspace solver for exponential
  integrators.
\newblock {\em J. Comput. Phys.}, 372:236--255, 2018.

\bibitem{HL03}
M.~Hochbruck and C.~Lubich.
\newblock On {M}agnus integrators for time--dependent {S}chr\"odinger
  equations.
\newblock {\em SIAM J. Numer. Anal.}, 41(3):945--963, 2003.

\bibitem{HO05bis}
M.~Hochbruck and A.~Ostermann.
\newblock Explicit exponential {R}unge--{K}utta methods for semilinear
  parabolic problems.
\newblock {\em SIAM J. Numer. Anal.}, 43(3):1069--1090, 2005.

\bibitem{HO10}
M.~Hochbruck and A.~Ostermann.
\newblock Exponential integrators.
\newblock {\em Acta Numer.}, 19:209--286, 2010.

\bibitem{KT05}
A.-K. Kassam and L.~N. Trefethen.
\newblock Fourth-order time-stepping for stiff {PDE}s.
\newblock {\em SIAM J. Sci. Comput.}, 26(4):1214--1233, 2005.

\bibitem{LPR19}
V.~T. Luan, J.~A. Pudykiewicz, and D.~R. Reynolds.
\newblock Further development of efficient and accurate time integration
  schemes for meteorological models.
\newblock {\em J. Comput. Phys.}, 376:817--837, 2019.

\bibitem{L95}
A.~Lunardi.
\newblock {\em Analytic {S}emigroups and {O}ptimal {R}egularity in {P}arabolic
  {P}roblems}.
\newblock Modern Birkh\"auser Classics. Birkh\"auser, Basel, 1 edition, 1995.

\bibitem{MLT17}
D.~L. Michels, V.~T. Luan, and M.~Tokman.
\newblock A stiffly accurate integrator for elastodynamic problems.
\newblock {\em ACM Trans. Graph.}, 36(4), 2017.

\bibitem{P83}
A.~Pazy.
\newblock {\em Semigroups of Linear Operators and Applications to Partial
  Differential Equations}, volume~44 of {\em Applied Mathematical Science}.
\newblock Springer, 1983.

\bibitem{SW09}
B.~Skaflestad and W.~M. Wright.
\newblock The scaling and modified squaring method for matrix functions related
  to the exponential.
\newblock {\em Appl. Numer. Math.}, 59(3--4):783--799, 2009.

\bibitem{T59}
H.~F. Trotter.
\newblock On the product of semi-groups of operators.
\newblock {\em Proc. Am. Math. Soc.}, 4(10):545--551, 1959.

\bibitem{ZZS23}
L.~Zhang, Q.~Zhang, and H.-W. Sun.
\newblock Preconditioned fourth-order exponential integrator for
  two-dimensional nonlinear fractional {G}inzburg-{L}andau equation.
\newblock {\em Comput. Math. Appl.}, 150:211--228, 2023.

\end{thebibliography}

\end{document}